\documentclass[reqno,a4paper,11pt]{article}
\usepackage{amsmath,amsthm,amssymb}
\usepackage[utf8]{inputenc}

\usepackage[T1]{fontenc}
\usepackage{hyperref}
\usepackage{comment}
\usepackage{enumerate}
\usepackage{bbm}
\usepackage{bm}
\usepackage{mathrsfs}
\usepackage{nicefrac}
\usepackage{accents}
\usepackage{verbatim}
\usepackage[a4paper]{geometry}
\usepackage{xcolor}
\usepackage{tikz}
\tikzset{every picture/.style={line width=0.75pt}}

\usepackage[sort&compress,numbers]{natbib}

\newtheorem{Thx}{Theorem}

\newtheorem{Corx}[Thx]{Corollary}

\newtheorem{theorem}{Theorem}[section]

\newtheorem{corollary}[theorem]{Corollary}
\newtheorem{proposition}[theorem]{Proposition}
\theoremstyle{definition}
\newtheorem{question}{Question}

\newtheorem{remark}[theorem]{Remark}
\newtheorem{example}[theorem]{Example}

\newcommand{\divides}{\mid}
\newcommand{\ndivides}{\nmid}

\newcommand{\cB}{{\mathcal B}}

\newcommand{\cF}{{\mathcal F}}

\newcommand{\cM}{{\mathcal M}}

\newcommand{\N}{{\mathbbm N}}

\newcommand{\Z}{{\mathbbm Z}}

\newcommand{\lcm}{\operatorname{lcm}}

\newcommand{\vep}{\varepsilon}
\newcommand{\beq}{\begin{equation}}

\newcommand{\eeq}{\end{equation}}

\makeatletter
\def\moverlay{\mathpalette\mov@rlay}
\def\mov@rlay#1#2{\leavevmode\vtop{%
   \baselineskip\z@skip \lineskiplimit-\maxdimen
   \ialign{\hfil$\m@th#1##$\hfil\cr#2\crcr}}}
\newcommand{\charfusion}[3][\mathord]{
    #1{\ifx#1\mathop\vphantom{#2}\fi
        \mathpalette\mov@rlay{#2\cr#3}
      }
    \ifx#1\mathop\expandafter\displaylimits\fi}
\makeatother

\providecommand{\noopsort}[1]{}
\author{Aurelia Dymek \and Stanis\l{}aw Kasjan \and Joanna Ku\l{}aga-Przymus}

\title{A note on $\mathscr{B}$-free sets and the existence of natural density}
\begin{document}
\maketitle
\begin{abstract}
Given $\mathscr{B}\subseteq\mathbb{N}$, let $\mathcal{M}_\mathscr{B}$ be its set of multiples; $\mathscr{B}$ is called Besicovitch if $\mathcal{M}_\mathscr{B}$ has natural density. With $\mathscr{B}$ one associates its tautification $\mathscr{B}'$, the unique taut set with the same Mirsky measure, and the minimal set $\mathscr{B}^*$, which determines the unique minimal subsystem of the $\mathscr{B}$-free subshift. It is known that if $\mathscr{B}'$ is Besicovitch, then so is $\mathscr{B}$. We show that this is the only restriction: all six combinations of the Besicovitch property for $\mathscr{B},\mathscr{B}',\mathscr{B}^*$ compatible with it occur. To this end we prove that every taut set is the tautification of a Besicovitch set, and that every minimal set is the minimal set associated with a Besicovitch taut set. We also show that $\mathcal{M}_{\mathscr{B}'}\setminus\mathcal{M}_\mathscr{B}$ may have positive upper Banach density even if $\mathscr{B}$ and $\mathscr{B}'$ are both Besicovitch in which case this set has zero natural density.
\end{abstract}
\tableofcontents
\section{Introduction}\label{intro}
\subsection{Setting}
\paragraph{Sets of multiples}
Let $\mathscr{B}\subseteq \mathbb{N}$ and let $\mathcal{M}_\mathscr{B}=\bigcup_{b\in\mathscr{B}}b\Z$ and $\mathcal{F}_\mathscr{B}=\Z\setminus \mathcal{M}_{\mathscr{B}}$ be the corresponding set of multiples and the $\mathscr{B}$-free set. The most prominent example of a $\mathscr{B}$-free set arises for $\mathscr{B}$ being the squares of all primes. Then $\mathcal{F}_\mathscr{B}$ is the so-called square-free set and its characteristic function $\mathbf{1}_{\mathcal{F}_\mathscr{B}}\in \{0,1\}^\Z$ restricted to~$\mathbb{N}$ is equal to the square of the M\"{o}bius function.

Natural density records the asymptotic proportion of a set in initial intervals $[1,N]$. In the simplest case when $\mathscr{B}$ is finite, $\mathcal{M}_\mathscr{B}$ is periodic and therefore has natural density. For infinite $\mathscr{B}$, the existence of natural density is no longer automatic. Davenport and Erd\H{o}s~\cite{MR43835} proved that the logarithmic density of $\mathcal{M}_\mathscr{B}$ (and thus also that of $\mathcal{F}_\mathscr{B}$) always exists; see Theorem~\ref{DE}. Besicovitch~\cite{MR1512943}, on the other hand, constructed sets of multiples without natural density. For a recent systematic comparison of different notions of density for classes of sets defined by multiplicative structure, including sets of multiples, see~\cite{Ross2026}.

We denote by $\delta$, $d$, $\underline{d}$ and $\overline{d}$ the logarithmic, natural, lower and upper natural density, respectively; see Section~\ref{prelim}. Whenever $d(\mathcal{M}_{\mathscr{B}})$ exists, we say that $\mathscr{B}$ is \textbf{Besicovitch}. We return to Besicovitch's construction and its modification in Section~\ref{triples}.

\paragraph{Dynamics}
The arithmetic set $\mathcal{F}_\mathscr{B}$ also gives rise to a symbolic dynamical system. Namely, let $\sigma\colon\{0,1\}^\Z\to \{0,1\}^\Z$ stand for the left shift, define
$\eta:=\mathbf{1}_{\mathcal{F}_\mathscr{B}}\in \{0,1\}^\Z$ and consider its orbit closure under $\sigma$:
\[
X_\eta:=\overline{\{\sigma^n \eta :n\in\Z\}}.
\]
We call $X_\eta$ the \textbf{$\mathscr{B}$-free system} (or the $\mathscr{B}$-free subshift). The system $X_\eta$ carries a distinguished $\sigma$-invariant measure, called the \textbf{Mirsky measure} and denoted by $\nu_\eta$. It describes the limiting frequencies of finite blocks in $\eta$ along sequences realizing the lower density of $\mathcal{M}_\mathscr{B}$; see Section~\ref{prelim}. The Mirsky measure need not be the only $\sigma$-invariant measure on $X_\eta$. Indeed, $\nu_\eta$ has zero entropy~\cite[Theorems~C and~F]{MR3803141}, so whenever $X_\eta$ has positive topological entropy, it supports invariant measures different from $\nu_\eta$; a broad class of such examples is provided by Erd\H{o}s $\mathscr{B}$-free systems~\cite{ELR}. For a subshift $X\subseteq\{0,1\}^{\mathbb Z}$, we denote by $\mathscr{M}(X)$ the
set of $\sigma$-invariant Borel probability measures on $X$.

\paragraph{Tautification and minimalization}
We say that $\mathscr{B}$ is \textbf{taut} if $\delta(\mathcal{M}_{\mathscr{B}\setminus \{b\}})<\delta(\mathcal{M}_{\mathscr{B}})$ for every $b\in\mathscr{B}$, and \textbf{minimal} if $\mathscr{B}$ is primitive (i.e.\ $b\divides b'$ for $b,b'\in\mathscr{B}$ implies $b=b'$) and $\eta$ is a Toeplitz sequence (i.e.\ for every $n\in\Z$ there exists $s_n$ such that $\eta$ is constant along $n+s_n\Z$).

With every $\mathscr{B}$ one associates two canonical sets. Its
\textbf{tautification} $\mathscr{B}'$ is the unique taut set with the same
Mirsky measure as $\mathscr{B}$~\cite[Theorem~C]{MR3803141}.
The \textbf{minimalization} of $\mathscr{B}$, denoted by
$\mathscr{B}^*$, is characterized by the fact that $X_{\eta^*}$ is the
unique minimal subsystem of $X_\eta$
\cite[Corollary~1.5 and Lemma~1.3]{MR4280951}.
Equivalently, $\mathscr{B}^*$ is the unique minimal set such that every
finite $0$--$1$ block occurring in the characteristic function of
$\mathcal{M}_{\mathscr{B}^*}$ also occurs in the characteristic function
of $\mathcal{M}_\mathscr{B}$; see Section~\ref{prelim} for the
corresponding finite-block characterization.
Roughly speaking, $\mathscr{B}'$ arises from $\mathscr{B}$ by adding certain $c\in\N$ for which $c\Z$ is contained in $\mathcal{M}_\mathscr{B}$ up to a set of zero logarithmic density, and $\mathscr{B}^*$ by adding certain $d\in\N$ for which arbitrarily long blocks of $d\Z$ appear, after a translation, in $\mathcal{M}_\mathscr{B}$; in both cases one then passes to primitive elements. Here and below, $\eta'=\mathbf{1}_{\mathcal{F}_{\mathscr{B}'}}$ and $\eta^*=\mathbf{1}_{\mathcal{F}_{\mathscr{B}^*}}$. The three systems form the chain $X_{\eta^*}\subseteq X_{\eta'}\subseteq X_\eta$; see~\eqref{j3}. Explicit formulas for $\mathscr{B}'$ and $\mathscr{B}^*$, together with the further properties used below, are collected in Section~\ref{prelim}.

\subsection{Questions and main results}
We ask how the existence of natural density behaves for $\mathscr{B}$, $\mathscr{B}'$ and $\mathscr{B}^*$. This property has a dynamical meaning as well: $\mathscr{B}$ is Besicovitch if and only if $\eta$ is a generic point for $\nu_\eta$. Indeed, if $\mathscr{B}$ is Besicovitch then $\eta$ is generic for $\nu_\eta$ by~\cite[Theorem~4.1]{MR3803141}; conversely, if $\eta$ is generic then the frequency of $1$s in $\eta$, i.e.\ $d(\mathcal{F}_\mathscr{B})$, exists.

To formulate the first question, with a set $\mathscr{B}$ we associate the triple $ijk\in\{0,1\}^3$ encoding the information whether $\mathscr{B}$, $\mathscr{B}'$, $\mathscr{B}^*$ is Besicovitch or not, e.g. 100 encodes the situation that $\mathscr{B}$ is Besicovitch whereas the remaining two sets are not.

In~\cite{MR1512943}, for each $\varepsilon>0$, Besicovitch constructed a set $\mathscr{B}\subseteq\N$ such that
\[
\underline{d}(\mathcal{M}_{\mathscr{B}})<\varepsilon
\qquad\text{and}\qquad
\overline{d}(\mathcal{M}_{\mathscr{B}})>1/2.
\]
As observed by Keller~\cite[Remark~2]{IrK}, such a set satisfies $\mathscr{B}^*=\{1\}$. On the other hand, it was proved in~\cite[Lemma~4.18]{MR3803141} that\footnote{This implication is misquoted in~\cite{MR3989121}; see Section~\ref{misquotation}.}
\begin{equation}\label{onlyobs}
\text{if $\mathscr{B}'$ is Besicovitch then $\mathscr{B}$ is Besicovitch.}
\end{equation}
Thus Besicovitch's construction yields the triple $001$.

Keller~\cite{IrK} subsequently constructed a non-Besicovitch minimal set $\mathscr{B}$. For this example $\mathscr{B}=\mathscr{B}'=\mathscr{B}^*$, and hence the corresponding triple is $000$. Finally, every Erd\H{o}s set (an infinite pairwise coprime set
$\mathscr{B}\subseteq\mathbb{N}\setminus\{1\}$
with $\sum_{b\in\mathscr{B}}1/b<\infty$) yields the triple $111$.

In view of~\eqref{onlyobs}, the triples $010$ and $011$ cannot occur, which leaves six of the eight triples. It is natural to ask the following.
\begin{question}\label{q1}
Do all triples $ijk\in\{0,1\}^3$ with $ij\neq 01$ occur?
\end{question}
It turns out that~\eqref{onlyobs} is the only obstruction, so the answer to Question~\ref{q1} is positive: all six remaining triples occur. The following general realization results, which we call \emph{Besicovitch detautification} and \emph{Besicovitch deminimalization}, allow us to prescribe the canonical sets: every taut set arises as the tautification of a Besicovitch set, and every minimal set arises as the minimal set associated with a Besicovitch taut set.
\begin{Thx}\label{taut_bes}
For any taut set $\mathscr{C}$ there exists a Besicovitch set $\mathscr{B}$ with $\mathscr{B}'=\mathscr{C}$.
\end{Thx}
\begin{Thx}\label{toeplitz_bes}
For any minimal set $\mathscr{D}$, there exists a Besicovitch taut set $\mathscr{B}$ with $\mathscr{B}^*=\mathscr{D}$.
\end{Thx}

By~\eqref{onlyobs}, the Besicovitch property passes from $\mathscr{B}'$ to $\mathscr{B}$. Theorem~\ref{taut_bes} shows that requiring $\mathscr{B}$ to be Besicovitch places no restriction on its tautification: every taut set can occur as $\mathscr{B}'$. Since the set $\mathscr{B}$ in Theorem~\ref{toeplitz_bes} is taut, we have $\mathscr{B}=\mathscr{B}'$, so Theorem~\ref{toeplitz_bes} produces triples of the form $11k$. Applied to Keller's non-Besicovitch minimal set and to (the tautification of) Besicovitch's example, Theorems~\ref{taut_bes} and~\ref{toeplitz_bes} yield the remaining triples $101$, $100$ and $110$. Table~\ref{tab:triples} summarizes how all six triples are realized. Moreover, each of them can be realized with $\mathscr{B}$, $\mathscr{B}'$ and $\mathscr{B}^*$ pairwise distinct (see the last column of Table~\ref{tab:triples}), so these phenomena are not caused by coincidences between the three sets.

\begin{table}[ht]
\centering
\begin{tabular}{c|l|l}
$ijk$ & realization & pairwise distinct $\mathscr{B},\mathscr{B}',\mathscr{B}^*$\\ \hline
$111$ & any Erd\H{o}s set & Example~\ref{4.1}\\
$001$ & Besicovitch~\cite{MR1512943} & Example~\ref{EX1}\\
$000$ & Keller~\cite{IrK} & Example~\ref{000}\\
$101$ & Thm~\ref{taut_bes} + Besicovitch & Example~\ref{ex101}\\
$100$ & Thm~\ref{taut_bes} + Keller & Example~\ref{EX2}\\
$110$ & Thm~\ref{toeplitz_bes} + Keller & Example~\ref{110}\\
$010$, $011$ & excluded by~\eqref{onlyobs} & ---
\end{tabular}
\caption{Density triples for
$\mathscr{B},\mathscr{B}',\mathscr{B}^*$.}\label{tab:triples}
\end{table}

In Example~\ref{ex101} we have $\nu_\eta=\nu_{\eta'}$ and $\mathscr{M}(X_\eta)=\mathscr{M}(X_{\eta'})$, but $\eta$ is generic for their common Mirsky measure whereas $\eta'$ is not. Thus the Besicovitch property is determined neither by the Mirsky measure nor by the set of invariant measures of the $\mathscr{B}$-free subshift; see Remark~\ref{dynamical_consequences_triples}.

The second question is motivated by a measure-theoretic problem: does the equality $\mathscr{M}(X_\eta)=\mathscr{M}(X_{\eta'})$ hold for every $\mathscr{B}\subseteq\N$? Here the relevant notion is upper Banach density, denoted by $d^*$. It looks for the largest proportions in translated intervals whose lengths tend to infinity; see Section~\ref{prelim}. Thus a set of zero natural density may still occupy a positive proportion of arbitrarily long intervals.

One possible approach, considered during the work on~\cite{DKPS}, is as follows. Every ergodic invariant measure of $X_\eta$ is a limit of empirical measures of $\eta$ along intervals whose lengths tend to infinity. Since $\{n:\eta(n)\neq\eta'(n)\}=\mathcal{M}_{\mathscr{B}'}\setminus\mathcal{M}_\mathscr{B}$, the equality $d^*(\mathcal{M}_{\mathscr{B}'}\setminus\mathcal{M}_\mathscr{B})=0$ would force every ergodic invariant measure of $X_\eta$ to be supported on $X_{\eta'}$; together with~\eqref{j3} this would give $\mathscr{M}(X_\eta)=\mathscr{M}(X_{\eta'})$. This leads to the following question.
\begin{question}\label{q2}
Is $d^*(\mathcal{M}_{\mathscr{B}'}\setminus \mathcal{M}_\mathscr{B})=0$ always true?
\end{question}
The answer to Question~\ref{q2} is negative. If one allows the tautification
to be non-Besicovitch, counterexamples follow already from positive upper
natural density, since $d^*\geq\overline{d}$:
\begin{Corx}\label{prop2}
For any non-Besicovitch taut set $\mathscr{C}$, there exists $\mathscr{B}$ such that $\mathscr{B}'=\mathscr{C}$ and
$\overline{d}(\mathcal{M}_{\mathscr{B}'} \setminus\mathcal{M}_{\mathscr{B}})>0$.
\end{Corx}
Corollary~\ref{prop2} follows from Theorem~\ref{taut_bes}, with $\mathscr{B}$ Besicovitch; see Remark~\ref{ThmA_implies_PropC}. Section~\ref{easy} also provides a shorter direct proof and shows that $\mathscr{B}$ can instead be chosen to be non-Besicovitch; see Propositions~\ref{bazujenaDE} and~\ref{prop:nonbes-choice}.

If $\mathscr{B}'$ is Besicovitch (and hence so is $\mathscr{B}$, by~\eqref{onlyobs}), then $\overline{d}(\mathcal{M}_{\mathscr{B}'}\setminus\mathcal{M}_{\mathscr{B}})=d(\mathcal{M}_{\mathscr{B}'})-d(\mathcal{M}_{\mathscr{B}})=0$, where the last equality follows from~\eqref{density_difference}. Thus, in this case, the set $\mathcal{M}_{\mathscr{B}'}\setminus\mathcal{M}_\mathscr{B}$ has zero natural density, but its upper Banach density may still be positive. Our main result concerning Question~\ref{q2} shows that this can indeed happen.
\begin{Thx}\label{prop1}
There exists a set $\mathscr{B}$ such that both $\mathscr{B}$ and $\mathscr{B}'$ are Besicovitch and $d^*(\mathcal{M}_{\mathscr{B}'} \setminus \mathcal{M}_{\mathscr{B}})>0$.
\end{Thx}

Theorem~\ref{prop1} shows that the strategy above for proving $\mathscr{M}(X_\eta)=\mathscr{M}(X_{\eta'})$ via vanishing upper Banach density cannot work in general, since $\mathcal{M}_{\mathscr{B}'}\setminus\mathcal{M}_{\mathscr{B}}$ may have positive upper Banach density. Nevertheless, in the example constructed there, $X_\eta$ and $X_{\eta'}$ have exactly the same invariant measures; see Remark~\ref{invariant_measures_example}. Thus vanishing upper Banach density of $\mathcal{M}_{\mathscr{B}'}\setminus\mathcal{M}_\mathscr{B}$ is sufficient, but not necessary, for the equality $\mathscr{M}(X_\eta)=\mathscr{M}(X_{\eta'})$. We do not know whether $\mathscr{M}(X_\eta)=\mathscr{M}(X_{\eta'})$ holds for every $\mathscr{B}\subseteq\mathbb{N}$. The equality does hold when $\eta^*$ is a regular Toeplitz sequence\footnote{Regularity means that the densities of the sets of $q$-periodic positions have supremum $1$ over $q\in\mathbb{N}$.}; see~\cite[Theorem~C]{DKPS}.

\subsection{Unions of rescaled patterns}\label{intro:patterns}
Both questions are addressed by a common construction, which replaces a prescribed taut or minimal set by unions of suitably rescaled patterns.

More precisely, we study sets $\mathscr{B}$ of the form
\begin{equation}\label{forma}
\mathscr{B}=\bigcup_{i\geq 1}\mathscr{E}_i\mathscr{A}_i,
\end{equation}
where $\mathscr{E}_i,\mathscr{A}_i\subset\N$ and, for $\mathscr{E},\mathscr{A}\subseteq\N$, we write $\mathscr{E}\mathscr{A}=\{ea : e\in\mathscr{E},\ a\in\mathscr{A}\}$ (and $c\mathscr{A}=\{c\}\mathscr{A}$ for $c\in\N$). We think of the sets $\mathscr{E}_i$ as ``scales'' and of the sets $\mathscr{A}_i$ as ``patterns'', and put $\mathscr{E}=\bigcup_{i\geq1}\mathscr{E}_i$.

We refer to the construction in~\eqref{forma} as the \emph{loosening procedure} when the scales are fixed and the patterns are chosen to obtain the required properties. We start with a prescribed set $\mathscr{E}$ and replace its elements by suitably chosen families of multiples, constructing a set $\mathscr{B}$ whose tautification or associated minimal set is exactly $\mathscr{E}$.

If $\mathscr{E}$ is taut, we choose Behrend patterns, whose sets of multiples have logarithmic density $1$. Then $\mathscr{B}'=\mathscr{E}$ (Theorem~\ref{pr:9_cor}): tautification recovers the prescribed set, even though $\mathscr{B}$ itself need not be taut.

If $\mathscr{E}$ is minimal, we choose infinite pairwise coprime patterns. Then $\mathscr{B}^*=\mathscr{E}$ (Theorem~\ref{wersja2_cor1}). Sufficiently sparse choices of these patterns also allow $\mathscr{B}$ to remain taut (Corollary~\ref{n2}), giving the deminimalization construction used in Theorem~\ref{toeplitz_bes}.

To obtain Theorems~\ref{taut_bes} and~\ref{toeplitz_bes}, we must additionally arrange for $\mathscr{B}$ to be Besicovitch. Theorem~\ref{jjj} supplies such pattern choices for any prescribed set of singleton scales $\mathscr{E}_i=\{e_i\}$.

Often,
\[
\mathscr{A}_i\subseteq\mathcal{P}_i\cap[K_i,\infty),
\]
where $\mathcal{P}_i$ is the set of primes in the arithmetic progression $2^{i+1}\Z+2^i+1$, i.e.\ the primes $p$ such that $p-1$ is divisible by $2^i$ but not by $2^{i+1}$, and $K_i$ is sufficiently large, while $\mathscr{E}$ is either taut or minimal. Appropriate choices of scales and patterns then control the corresponding properties of $\mathscr{B}$, $\mathscr{B}'$ and $\mathscr{B}^*$. The sets $\mathcal{P}_i$, $i\geq1$, partition the set of odd primes. Each of them is a Behrend set (by Dirichlet's theorem and \cite[formula~(0.69)]{MR1414678}), while its relative density among the primes is $2^{-i}$; this decay is used in the proof of Theorem~\ref{jjj} to verify the Besicovitch property of $\mathscr{B}$.

\subsection{Organization of the paper}
Section~\ref{prelim} collects the background on densities, the classes of sets used below, the Mirsky measure, the tautification and the minimalization. In Section~\ref{des} we develop the loosening procedure (Theorems~\ref{pr:9_cor} and~\ref{wersja2_cor1}). In Section~\ref{se4} we prove Theorem~\ref{jjj} and Corollaries~\ref{n1} and~\ref{n2}, and derive Theorems~\ref{taut_bes} and~\ref{toeplitz_bes}. In Section~\ref{triples} we combine these results with the known examples to realize all six admissible triples, and we give examples in which $\mathscr{B}$, $\mathscr{B}'$ and $\mathscr{B}^*$ are pairwise distinct. Finally, in Section~\ref{easy} we prove Corollary~\ref{prop2} in a refined form and complete the proof of Theorem~\ref{prop1}.

\section{Preliminaries}\label{prelim}
Throughout the paper, $\mathbb{N}=\{1,2,3,\ldots\}$, and $\mathcal{P}$ denotes the set of all primes.

\subsection{Densities and arithmetic classes}

\paragraph{Densities}
Given $A\subseteq \Z$ satisfying $A=-A$ (or simply $A\subseteq \N$), let $\delta(A)$, $d(A)$, $\underline{d}(A)$, $\overline{d}(A)$ and $d^*(A)$ stand, respectively, for the logarithmic density, natural density, lower natural density, upper natural density and upper Banach density of $A\cap \N$:
\[
\delta(A)=\lim_{n\to\infty}\frac{1}{\ln n}\sum_{1\leq k\leq n,\ k\in A}\frac{1}{k} \text{ and }d(A)=\lim_{n\to\infty}\frac{1}{n}|A\cap [1,n]|
\]
(with obvious modifications for $\underline{d}$ and $\overline{d}$), and
\[
d^*(A)=\lim_{n\to\infty}\sup_{m\geq 0}\frac{1}{n}|A\cap [m+1,m+n]|.
\]
Davenport and Erd\H{o}s~\cite{MR43835} proved that the logarithmic density of $\mathcal{M}_\mathscr{B}$ (and thus also that of $\mathcal{F}_\mathscr{B}$) always exists and it is approximated by the densities of the natural periodic approximations of $\mathcal{M}_\mathscr{B}$. Namely, we have the following.
\begin{theorem}[Davenport--Erd\H{o}s~\cite{MR43835}]\label{DE}
For any $\mathscr{B}\subseteq \N$,
\[
\delta(\mathcal{M}_{\mathscr{B}})=\underline{d}(\mathcal{M}_\mathscr{B})=\lim_{K\to \infty}d(\mathcal{M}_{\mathscr{B}\cap[1,K]}).
\]
\end{theorem}

\paragraph{Classes of sets}
The topological and measure-theoretic dynamics of $X_\eta$ depend strongly on
the arithmetic properties of $\mathscr{B}$. Besides the notions of primitive,
taut, minimal\footnote{Note that this includes the case of periodic $\eta$.
The term \emph{minimal set} is new, while the concept itself appeared earlier,
cf.~\cite{DKK,MR3947636}.} and Erd\H{o}s sets introduced in
Section~\ref{intro}, we will also use two further classes. We say that
$\mathscr{B}$ is:
\begin{itemize}
\item \textbf{thin} if $\sum_{b\in\mathscr{B}}\nicefrac{1}{b}<\infty$,
\item \textbf{Behrend} if $\mathscr{B}\subseteq\N\setminus\{1\}$ and
$\delta(\mathcal{M}_\mathscr{B})=1$.
\end{itemize}
Usually we will assume that $\mathscr{B}$ is primitive. This assumption has no
influence on the dynamics since
$\mathcal{M}_\mathscr{B}=\mathcal{M}_{\mathscr{B}^{prim}}$, where
$\mathscr{B}^{prim}=\mathscr{B}\setminus\bigcup_{k\geq2}k\mathscr{B}$.\footnote{Note that $\mathscr{B}^{prim}$ is the unique primitive set with the same set of multiples as that of $\mathscr{B}$.}

\paragraph{Basic relations between classes.}
Every thin set is Besicovitch (this follows, e.g., from Theorem~\ref{DE}).
Every primitive thin set is taut~\cite[implication~(7)]{MR3803141}; in
particular, any Erd\H{o}s set is taut. Any set of primes with nonsummable
series of reciprocals is Behrend~\cite[formula~(0.69)]{MR1414678}, while
any minimal set is taut~\cite[Proposition~5.3]{MR3731019}. A finite union
of Besicovitch sets is Besicovitch~\cite[Theorem~1.6]{MR1414678}.
Moreover, any tail of a Behrend set remains Behrend
\cite[Corollary~0.14]{MR1414678}, and if
$\mathscr{B}=\mathscr{B}_1\cup\mathscr{B}_2$ is Behrend, then at least one
of $\mathscr{B}_1,\mathscr{B}_2$ is Behrend
\cite[Corollary~0.14]{MR1414678}.

\paragraph{Taut sets.}
A set $\mathscr{B}$ is taut if and only if it is primitive and
$c\mathscr{A}\not\subseteq\mathscr{B}$ for any $c\in\N$ and any Behrend set
$\mathscr{A}$~\cite[Corollary~0.19]{MR1414678}. A subset of a taut set is
taut. If $\mathscr{B}=\mathscr{B}_1\cup\mathscr{B}_2$ is primitive and both
$\mathscr{B}_1$ and $\mathscr{B}_2$ are taut, then $\mathscr{B}$ is taut.

\paragraph{Minimal sets.}
A set $\mathscr{B}$ is minimal if and only if it is primitive and
$c\mathscr{A}\not\subseteq\mathscr{B}$ for any $c\in\N$ and any infinite
pairwise coprime set $\mathscr{A}$~\cite[Theorem~B]{MR3947636}. A subset
of a minimal set is minimal. If
$\mathscr{B}=\mathscr{B}_1\cup\mathscr{B}_2$ is primitive and both
$\mathscr{B}_1$ and $\mathscr{B}_2$ are minimal, then $\mathscr{B}$ is
minimal.
\subsection{The Mirsky measure and canonical sets}

\paragraph{The Mirsky measure}
A simple way to define the Mirsky measure $\nu_\eta$ is to use the sequence $(N_k)$ realizing the lower density of $\mathcal{M}_{\mathscr{B}}$ as in Theorem~\ref{DE}. Namely, by~\cite[Theorem~4.1]{MR3803141},
\begin{equation}\label{qgen}
\eta \text{ is a quasi-generic point, i.e.\ the weak-* limit }\lim_{k\to\infty}\frac{1}{N_k}\sum_{n\leq N_k}\delta_{\sigma^n \eta} \text{ exists}.
\end{equation}
The resulting measure is the Mirsky measure $\nu_\eta$. We return to the other invariant measures on $X_\eta$, i.e.\ to $\mathscr{M}(X_\eta)$, below.

\paragraph{Tautification and minimalization}
We now recall in more detail the two canonical sets $\mathscr{B}'$ and $\mathscr{B}^*$. Recall that the class of taut sets was first introduced in a number-theoretic context~\cite{MR1414678}. Later it turned out that they are important also from the dynamical point of view~\cite{MR3803141}. The tautification introduced in Section~\ref{intro} satisfies $\eta'\leq\eta$~\cite[Theorem~C]{MR3803141}.

By~\cite{MR3803141,DKKu},
\[
\mathscr{B}'=(\mathscr{B}\cup C_\mathscr{B})^{prim},
\]
where
\begin{equation}\label{skaleC}
C_\mathscr{B}=\{c\in\N : c\mathscr{A}\subseteq \mathscr{B}\text{ for some Behrend set }\mathscr{A}\}.
\end{equation}
Intuitively, since $\delta(\mathcal{M}_\mathscr{A})=1$ for a Behrend set $\mathscr{A}$, $c\in C_\mathscr{B}$ implies that $c\Z$ is contained in $\mathcal{M}_\mathscr{B}$ up to a set of zero logarithmic density. Moreover,
\begin{equation}\label{density_difference}
\underline{d}(\mathcal{M}_{\mathscr{B}})
=
\underline{d}(\mathcal{M}_{\mathscr{B}'}).
\end{equation}
Furthermore, if $(N_k)_{k\geq1}$ realizes the lower density of $\mathcal{M}_{\mathscr{B}'}$, then it also realizes the lower density of $\mathcal{M}_{\mathscr{B}}$ and
\[
\lim_{k\to\infty}\frac{1}{N_k}
\big|(\mathcal{M}_{\mathscr{B}'}\setminus\mathcal{M}_{\mathscr{B}})\cap[1,N_k]\big|=0.
\]

The second canonical set is
\[
\mathscr{B}^*=(\mathscr{B}\cup D_\mathscr{B})^{prim},
\]
where
\begin{equation}\label{skaleD}
D_\mathscr{B}=\{d\in \N : d\mathscr{A}\subseteq\mathscr{B}\text{ for some infinite pairwise coprime set }\mathscr{A}\}.
\end{equation}
By the Chinese remainder theorem, $d\in D_\mathscr{B}$ implies that arbitrarily long blocks of $d\Z$ appear, after a translation, in $\mathcal{M}_\mathscr{B}$. Since every Behrend set contains an infinite pairwise coprime set~\cite{MR3803141}, we have $C_\mathscr{B}\subseteq D_\mathscr{B}$ and hence
\[
\eta^*\leq\eta'\leq\eta,
\]
where $\eta^*=\mathbf{1}_{\mathcal{F}_{\mathscr{B}^*}}$.

\subsection{Dynamical characterizations}
The following result shows that a taut set is uniquely determined by its associated subshift or Mirsky measure.
\begin{theorem}[{\cite[Theorem~L]{MR3803141}}]\label{ThmL}
Let $\mathscr{B}_1,\mathscr{B}_2\subseteq \N$ be taut. Then the following conditions are equivalent: $\mathscr{B}_1=\mathscr{B}_2$ $\iff$ $X_{\eta_1}=X_{\eta_2}$ $\iff$ $\nu_{\eta_1}=\nu_{\eta_2}$, where $\eta_i=\mathbf{1}_{\mathcal{F}_{\mathscr{B}_i}}$ for $i=1,2$.
\end{theorem}

Recall that every $\mathscr{B}$-free subshift is \textbf{essentially minimal}~\cite[Theorem~A]{MR3803141}, and its unique minimal subsystem is $X_{\eta^*}$~\cite[Corollary~1.5 and Lemma~1.3]{MR4280951}. Moreover, $X_\eta$ is minimal if and only if $\eta$ is a Toeplitz sequence~\cite[Theorem~3.7]{DKKu}; see also~\cite[Theorem~B]{MR3947636}. This explains the terminology ``minimal set''. For every $\mathscr{B}$ we have
\begin{equation}\label{j3}
X_{\eta^*}\subseteq X_{\eta'}\subseteq X_\eta.
\end{equation}
The first inclusion follows from~\cite[Remark~3.22]{DKKu}, and the second from~\cite[(27)]{MR4289651}.

\paragraph{Taut subsystems}
For $\mathscr{C}\subseteq\N$ we write $\eta_\mathscr{C}=\mathbf{1}_{\mathcal{F}_\mathscr{C}}$. For taut sets $\mathscr{C}_1,\mathscr{C}_2$ write
\[
\mathscr{C}_1\prec\mathscr{C}_2
\quad\Longleftrightarrow\quad
X_{\eta_{\mathscr{C}_1}}\subseteq X_{\eta_{\mathscr{C}_2}},
\]
and set
\[
\operatorname{Taut}(\mathscr{B})
:=
\{\mathscr{C}\subseteq\N:
\mathscr{C}\text{ is taut and }X_{\eta_\mathscr{C}}\subseteq X_\eta\}.
\]
\begin{theorem}[see {\cite[Proposition~D]{DKPS}}, cf.\ Theorem~\ref{ThmL}]\label{taut_more}
Let $\mathscr{B}\subseteq\N$. Then for any taut $\mathscr{C}\subseteq\N$ we have: $\eta^*\leq\eta_\mathscr{C}\leq\eta$ if and only if $X_{\eta_\mathscr{C}}\subseteq X_\eta$.
\end{theorem}

Minimal sets are precisely the minimal elements with respect to $\prec$. For every $\mathscr{B}$, the sets $\mathscr{B}^*$ and $\mathscr{B}'$ are respectively the smallest and the largest elements of $\operatorname{Taut}(\mathscr{B})$; see the discussion at the end of Section~1.3.2 in~\cite{DKPS}. In particular, $\mathscr{B}^*$ is uniquely determined by $\mathcal{M}_{\mathscr{B}}$, and hence
\[
\mathscr{B}^*=(\mathscr{B}^{prim})^*.
\]

Moreover, $(\mathscr{B}')^*=\mathscr{B}^*$; see \cite[(12)]{DKPS}.

\paragraph{Finite-block characterization of the minimalization}
We also need a description of $\mathscr{B}^*$ in terms of finite configurations, which we obtain in two steps.

First, $\mathscr{B}^*$ is the unique minimal set $\mathscr{C}$ such that $X_{\eta_\mathscr{C}}\subseteq X_\eta$. Indeed, $\mathscr{B}^*$ is minimal and $X_{\eta^*}\subseteq X_\eta$ by~\eqref{j3}. Conversely, let $\mathscr{C}$ be a minimal set with $X_{\eta_\mathscr{C}}\subseteq X_\eta$. Since $\eta_\mathscr{C}$ is a Toeplitz sequence, its orbit closure $X_{\eta_\mathscr{C}}$ is a minimal subshift, so it is a minimal subsystem of $X_\eta$. As $X_{\eta^*}$ is the unique minimal subsystem of $X_\eta$~\cite[Corollary~1.5 and Lemma~1.3]{MR4280951}, we obtain $X_{\eta_\mathscr{C}}=X_{\eta^*}$. Both $\mathscr{C}$ and $\mathscr{B}^*$ are minimal, hence taut by \cite[Proposition~5.3]{MR3731019}, and Theorem~\ref{ThmL} yields $\mathscr{C}=\mathscr{B}^*$.

Second, for a minimal set $\mathscr{C}$, the inclusion $X_{\eta_\mathscr{C}}\subseteq X_\eta$ holds if and only if for every $n\in\N$ there exists $k\in\Z$ such that
\begin{equation}\label{windows}
\mathcal{M}_{\mathscr{C}}\cap[0,n]=\mathcal{M}_{\mathscr{B}}\cap[k,k+n]-k,
\end{equation}
i.e.\ $\eta_\mathscr{C}|_{[0,n]}=\eta|_{[k,k+n]}$. Indeed, the inclusion $X_{\eta_\mathscr{C}}\subseteq X_\eta$ means that $\eta_\mathscr{C}\in X_\eta$, that is, every block $\eta_\mathscr{C}|_{[-n,n]}$, $n\in\N$, appears in $\eta$. In particular, every block $\eta_\mathscr{C}|_{[0,n]}$ appears in $\eta$, which is~\eqref{windows}. Conversely, assume~\eqref{windows} for every $n\in\N$ and fix $n\in\N$. Since $\eta_\mathscr{C}$ is a Toeplitz sequence, for every $j\in[-n,n]$ there exists $s_j\in\N$ such that $\eta_\mathscr{C}$ is constant along $j+s_j\Z$. Let $s=\operatorname{lcm}\{s_j : -n\leq j\leq n\}$ and choose $m\in\N$ with $ms\geq n$. Then $\eta_\mathscr{C}(j+ms)=\eta_\mathscr{C}(j)$ for every $j\in[-n,n]$, so the block $\eta_\mathscr{C}|_{[-n,n]}$ coincides with $\eta_\mathscr{C}|_{[ms-n,ms+n]}$, which is a part of $\eta_\mathscr{C}|_{[0,ms+n]}$. By~\eqref{windows} applied with $ms+n$ in place of $n$, the block $\eta_\mathscr{C}|_{[0,ms+n]}$ appears in $\eta$, and hence so does $\eta_\mathscr{C}|_{[-n,n]}$. As $n$ was arbitrary, $\eta_\mathscr{C}\in X_\eta$.

Combining the two steps, we obtain the following concrete characterization: $\mathscr{B}^*$ is the unique minimal set $\mathscr{C}$ such that for every $n\in\N$ there exists $k\in\Z$ satisfying~\eqref{windows}.

\paragraph{Invariant measures}
The role of $\mathscr{B}^*$ is not only topological. As shown in~\cite[Theorem~C]{DKPS}, if $\eta^*$ is a regular Toeplitz sequence (see Section~\ref{intro}), then
\[
\mathscr{M}(X_\eta)=\mathscr{M}(X_{\eta'}).
\]
We use this fact in Remark~\ref{invariant_measures_example}.

\section{Loosening procedure}\label{des}
This section develops the common loosening mechanism. We first state the two consequences used later for Questions~\ref{q1} and~\ref{q2}, and then prove them from the structural propositions below.
\begin{theorem}\label{pr:9_cor}
Let $\mathscr{B}=\bigcup_{i\geq 1}\mathscr{E}_i\mathscr{A}_i$, where $\mathscr{E}=\bigcup_{i\geq 1}\mathscr{E}_i$ is taut and $\mathscr{A}_i$ are Behrend sets. Then $\mathscr{B}'=\mathscr{E}$.
\end{theorem}
\begin{theorem}\label{wersja2_cor1}
Let $\mathscr{B}=\bigcup_{i\geq 1}\mathscr{E}_i\mathscr{A}_i$, where $\mathscr{E}=\bigcup_{i\geq 1}\mathscr{E}_i$ is minimal and $\mathscr{A}_i$ are infinite pairwise coprime sets. Then $\mathscr{B}^*=\mathscr{E}$.
\end{theorem}

\paragraph{Tautification}
We use the order on taut sets and Theorem~\ref{taut_more} recalled in Section~\ref{prelim}.
Let $\mathscr{B}\subseteq \N$ and let $C_\mathscr{B}$ be the set defined in~\eqref{skaleC}, so that $\mathscr{B}'=(\mathscr{B}\cup C_\mathscr{B})^{prim}$.
\begin{proposition}\label{pr:9}
Let $\mathscr{C}\subseteq C_\mathscr{B}$ and $(\mathscr{B} \cup \mathscr{C})^{prim}$ be taut. Then $\mathscr{B}'=(\mathscr{B} \cup \mathscr{C})^{prim}$.
\end{proposition}
\begin{proof}
By use of the fact that passing to the set of the primitive elements does not change the set of the multiples of a given set of natural numbers, we obtain
\begin{equation}\label{totaut}
\eta^*\leq\eta' =\eta_{\mathscr{B}\cup C_\mathscr{B}} \leq \eta_{\mathscr{B}\cup\mathscr{C}}=\eta_{(\mathscr{B}\cup\mathscr{C})^{prim}}\leq \eta.
\end{equation}
\begin{sloppypar}
It follows by Theorem~\ref{taut_more} that $X_{\eta_{(\mathscr{B}\cup\mathscr{C})^{prim}}}\subseteq X_\eta$. However, since $\mathscr{B}'$ is the largest element in $\text{Taut}(\mathscr{B})$, we have $(\mathscr{B}\cup \mathscr{C})^{prim}\prec \mathscr{B}'$, i.e.\ $X_{\eta_{(\mathscr{B}\cup\mathscr{C})^{prim}}}\subseteq X_{\eta'}$. Applying again Theorem~\ref{taut_more}, we obtain $\eta_{(\mathscr{B}\cup \mathscr{C})^{prim}}\leq \eta'$. Thus, $\eta_{(\mathscr{B}\cup \mathscr{C})^{prim}}= \eta'$. It remains to use Theorem~\ref{ThmL} to complete the proof.
\end{sloppypar}
\end{proof}
\begin{proof}[Proof of Theorem~\ref{pr:9_cor}]
Take $\mathscr{B}=\bigcup_{i\geq 1}\mathscr{E}_i\mathscr{A}_i$, where $\mathscr{E}=\bigcup_{i\geq 1}\mathscr{E}_i$ is taut and $\mathscr{A}_i$ are Behrend sets. It suffices to notice that $\mathscr{E}\subseteq C_\mathscr{B}$ and $(\mathscr{B}\cup\mathscr{E})^{prim}=\mathscr{E}^{prim}=\mathscr{E}$ (the last equality follows directly by the fact that $\mathscr{E}$ is taut and hence primitive). To obtain $\mathscr{B}'=\mathscr{E}$ it suffices to apply Proposition~\ref{pr:9}.
\end{proof}
\paragraph{Minimalization}
Let $D_\mathscr{B}$ be the set defined in~\eqref{skaleD}, so that $\mathscr{B}^*=(\mathscr{B}\cup D_\mathscr{B})^{prim}$.
\begin{proposition}\label{wersja2}
Let $\mathscr{D}\subseteq D_\mathscr{B}$ and $(\mathscr{B}\cup \mathscr{D})^{prim}$ be minimal. Then $(\mathscr{B}\cup\mathscr{D})^{prim}=\mathscr{B}^*$.
\end{proposition}
\begin{proof}
Similarly as in the proof of Proposition~\ref{pr:9}, we have
\[\eta_{\mathscr{B}^*}=\eta_{\mathscr{B}\cup D_\mathscr{B}}\leq \eta_{\mathscr{B}\cup\mathscr{D}}=\eta_{(\mathscr{B}\cup\mathscr{D})^{prim}}\leq \eta.\]
Since $(\mathscr{B}\cup\mathscr{D})^{prim}$ is minimal, it is taut by \cite[Proposition~5.3]{MR3731019}. Hence the preceding inequality and Theorem~\ref{taut_more} give
$X_{\eta_{(\mathscr{B}\cup\mathscr{D})^{prim}}}\subseteq X_\eta$.
Both $(\mathscr{B}\cup\mathscr{D})^{prim}$ and $\mathscr{B}^*$ are minimal, while $X_\eta$ is essentially minimal. Therefore
$X_{\eta_{(\mathscr{B}\cup\mathscr{D})^{prim}}}=X_{\eta_{\mathscr{B}^*}}$, and Theorem~\ref{ThmL} yields $(\mathscr{B}\cup\mathscr{D})^{prim}=\mathscr{B}^*$.
\end{proof}
\begin{proof}[Proof of Theorem~\ref{wersja2_cor1}]
Let $\mathscr{B}=\bigcup_{i\geq 1}\mathscr{E}_i\mathscr{A}_i$, where $\mathscr{E}=\bigcup_{i\geq 1}\mathscr{E}_i$ is minimal and $\mathscr{A}_i$ are infinite pairwise coprime sets. It suffices to notice that $\mathscr{E}\subseteq D_\mathscr{B}$ and $(\mathscr{B}\cup\mathscr{E})^{prim}=\mathscr{E}^{prim}=\mathscr{E}$ (again, the last equality follows directly by the fact that $\mathscr{E}$ is minimal and hence taut, so it is primitive).
The assertion follows by Proposition~\ref{wersja2}.
\end{proof}
\section{Besicovitch detautification and deminimalization}\label{se4}
The results of this section provide the refined forms of Theorems~\ref{taut_bes} and~\ref{toeplitz_bes} used to complete Question~\ref{q1} in Section~\ref{triples}. We first prove the Besicovitch realization theorem that underlies both constructions.

\subsection{A Besicovitch realization theorem}

\begin{theorem}\label{jjj}
For any $\mathscr{E}=\bigcup_{i\geq 1}\{e_i\}$, there exists a sequence $(K_i)_{i\geq1}$ of positive integers such that if $\emptyset\neq\mathscr{A}_i\subseteq \mathcal{P}_i\cap [K_i,\infty)$, where $\mathcal{P}_i$ stands for the set of primes in $2^{i+1}\Z+2^i+1$, then $\mathscr{B}=\bigcup_{i\geq 1}e_i\mathscr{A}_i$ is Besicovitch. If we additionally assume that $\mathscr{E}$ is primitive then $K_i$ can be chosen so that $\mathscr{B}$ is primitive.
\end{theorem}
To prove Theorem~\ref{jjj}, we first prepare two tools.
\paragraph{Erd\H{o}s' criterion}
To verify the Besicovitch property we use the following criterion of Erd\H{o}s~\cite{MR0026088} (we recall here the formulation from \cite[Theorem 0]{ErHaTe}).
\begin{theorem}[Erd\H{o}s~\cite{MR0026088}]\label{besicovitch}
Let $\mathscr{B}\subseteq \N$. Then $d(\mathcal{M}_\mathscr{B})$ exists if and only if
\[
\lim\limits_{\varepsilon\rightarrow 0}\limsup_{x\to \infty}x^{-1}\sum_{a\in\mathscr{B}\cap(x^{1-\varepsilon},x]}|[1,x]\cap a\Z \cap \mathcal{F}_{\mathscr{B}\cap [1,a)}|=0.
\]
\end{theorem}
\paragraph{A Mertens estimate in arithmetic progressions}
We will use the following Mertens-type estimate for primes in arithmetic progressions:
\begin{equation}\label{eq:Prachar}
\sum_{p\le x,\ p\equiv l\pmod{k}}\frac{1}{p}
=
\frac{1}{\varphi(k)}\ln\ln x + B_{k,l} + \textrm{O}\left(\frac{1}{\ln x}\right)
\end{equation}
for any coprime $k,l$, where $\varphi$ denotes the Euler totient function, $B_{k,l}$ is a constant depending on $k$ and $l$, and the implied constant depends only on $k$.

For context, Mertens~\cite{Mertens1874} proved the classical formulas
\begin{align*}
\sum_{p\leq x}\frac{\ln p}{p}&=\ln x+\textrm{O}(1),\\
\sum_{p\leq x}\frac{1}{p}&=\ln \ln x+M+\textrm{O}\left(\frac{1}{\ln x}\right),\\
\prod_{p\leq x}\left(1-\frac{1}{p}\right)&=\frac{e^{-\gamma}}{\ln x}\left(1+\textrm{o}(1)\right),
\end{align*}
where $M$ stands for the Meissel--Mertens constant and $\gamma$ is the Euler constant. These formulas have analogues for primes in arithmetic progressions.
Williams~\cite[Theorem~1]{MR0364137}
established an analogue of the third formula for
\begin{equation}\label{williams}
\prod_{p\le x,\ p\equiv l\pmod{k}}\left(1-\frac{1}{p}\right)
\end{equation}
for any coprime $k,l$. Formula~\eqref{eq:Prachar}, which is the analogue of the second Mertens formula needed here, can also be deduced from Williams' result. A derivation from~\eqref{williams}, together with more information on the constants $B_{k,l}$, can be found in~\cite[formula~(1)]{MR2743571}; see also~\cite[formula~(6)]{MR4750804}.

The proof of Theorem~\ref{jjj} begins by estimating the reciprocal sums appearing in Erd\H{o}s' criterion. The cutoffs $K_i$ make the required bounds valid for all $x\geq1$. We then split the sum over $i$ into a finite part and a summable tail. Primitivity is verified separately at the end.

\begin{proof}[Proof of Theorem~\ref{jjj}]
Let $\mathscr{E}=\bigcup_{i\geq 1}\{e_i\}$ and let $\emptyset \neq \mathscr{A}_i\subseteq \mathcal{P}_i\cap [K_i,\infty)$, where $K_i$ will be chosen recursively,
depending on the previously fixed $e_i$. We will use Theorem~\ref{besicovitch} to show that   $\mathscr{B}$ is Besicovitch. To this end, notice that for any $x\geq 1$ and any $\varepsilon>0$, we have
\begin{align}
\begin{split}\label{eq:bound2}
\sum_{a\in \mathscr{B} \cap(x^{1-\varepsilon}, x]}|[1,x]\cap a\Z\cap \mathcal{F}_{\mathscr{B}\cap [1,a)}|
&\le \sum_{a\in \mathscr{B} \cap(x^{1-\varepsilon}, x]}|[1,x]\cap a\Z|\\
&\le \sum_{i=1}^{\infty}\sum_{p\in \mathscr{A}_i\cap (x^{1-\varepsilon}/e_i,x/e_i]}|[1,x]\cap e_ip\Z|
\end{split}
\end{align}

\smallskip\noindent\emph{Reciprocal-sum estimates.}
By the obvious estimate $|[1,x]\cap a\Z|\le \frac{x}{a}$, we have
\begin{equation}\label{eq:estim1}
\begin{aligned}
&\sum_{p\in \mathscr{A}_i\cap (x^{1-\varepsilon}/e_i,x/e_i]}|[1,x]\cap e_ip\Z|\\
&\qquad\le x\left(\sum_{p\in \mathscr{A}_i\cap(x^{1-\varepsilon}/e_i, x/e_i]}\frac{1}{e_i p}\right)\\
&\qquad=\frac{x}{e_i}g_{i,\varepsilon}(x)\le xg_{i,\varepsilon}(x),
\end{aligned}
\end{equation}
where
\[
g_{i,\varepsilon}(x):=\sum_{p\in \mathscr{A}_i\cap\left(\frac{x^{1-\varepsilon}}{e_i},\frac{x}{e_i}\right]}\frac{1}{p}.
\]
for every $\varepsilon>0$ and $x\geq 1$.

From now on, we will assume that $\varepsilon\in (0,\nicefrac12]$. Since $\varphi(2^{i+1})=2^i$, it follows from~\eqref{eq:Prachar} that
\begin{equation}\label{eq:between}
g_{i,\varepsilon}(x)\le \frac{1}{2^i}\left(\ln\ln\frac{x}{e_i}-\ln\ln\frac{x^{1-\varepsilon}}{e_i}\right)+\textrm{O}\left(\frac{1}{\ln x}\right).
\end{equation}
The constant in $\textrm{O}(\frac{1}{\ln x})$ depends on $i$, but not on the choice of $K_i$ and not on $\varepsilon$.

We will show that
\begin{equation}\label{eq:nierownosc}
\ln\ln \frac{x}{e_i}-\ln\ln \frac{x^{1-\varepsilon}}{e_i}\leq 2\varepsilon(1+o(1)),
\end{equation}
where $\rm{o}(1)$ depends on $i$. Indeed, notice that
\[
\begin{aligned}
\ln\ln \frac{x}{e_i}-\ln\ln \frac{x^{1-\varepsilon}}{e_i}
&=\ln (y-s_i)-\ln(1-\varepsilon)-\ln\left(y-\frac{1}{1-\varepsilon}s_i\right)\\
&=\ln\left(1+\frac{\frac{\varepsilon}{1-\varepsilon}s_i}{y-\frac{1}{1-\varepsilon}s_i}\right)-\ln(1-\varepsilon),
\end{aligned}
\] where $y=\ln x$ and $s_i=\ln e_i$. Moreover,
\[
\ln\left(1+\frac{\frac{\varepsilon}{1-\varepsilon}s_i}{y-\frac{1}{1-\varepsilon}s_i}\right)\leq \varepsilon\frac{\frac{1}{1-\varepsilon}s_i}{y-\frac{1}{1-\varepsilon}s_i}\leq \varepsilon\frac{2s_i}{y-2s_i}=\varepsilon o(1),
\]
as $x\rightarrow\infty$, where $o(1)$ depends on $i$ and the first inequality is a consequence of the inequality $\ln(1+h)\leq h$ and the second one holds by $\varepsilon\in (0,\nicefrac{1}{2}]$. To finish the proof of~\eqref{eq:nierownosc}, we use
\begin{equation}\label{eq:star}
-2h<\ln(1-h)<-h
\end{equation}
for $0<h\leq\frac{1}{2}$ (the inequalities in~\eqref{eq:star} follow directly by comparing derivatives).

Since $0<\varepsilon\leq\frac{1}{2}$, then, by \eqref{eq:between} and \eqref{eq:nierownosc}
\begin{equation}\label{eq:nierownosc2}
g_{i,\varepsilon}(x)\le \frac{1}{2^{i-1}}\varepsilon(1+o(1)),
\end{equation}
as $x\rightarrow \infty$, where, for fixed $i$ and $\varepsilon$, the term $o(1)$ tends to zero and does not depend on $K_i$.

\smallskip\noindent\emph{Choice of the cutoffs and the tail estimate.}
Take $\varepsilon=\frac{1}{2^i}$. Then
\[
g_{i,\frac{1}{2^i}}(x)\leq \frac{1}{4^{i-1}}\cdot\frac{1}{2}(1+\textrm{o}(1)).
\]
It follows that there exists $x_i\geq 1$ such that for $x\geq x_i$ we have
\begin{equation}\label{eq:nierownosc3}
g_{i,\frac{1}{2^i}}(x)\leq \frac{1}{4^{i-1}}.
\end{equation}
If we assume that $K_i\geq \frac{x_i}{e_i}$ then $g_{i,\frac{1}{2^i}}(x)=0$ for $x<x_i$ by the definition of $g_{i,\frac{1}{2^i}}$ as then $\mathscr{A}_i\cap \left(\frac{x^{1-\varepsilon}}{e_i},\frac{x}{e_i} \right]=\emptyset$. Thus, under this choice of the parameters,~\eqref{eq:nierownosc3} holds for every $x$ in this case.

From now on, we fix $K_i> \max\left\{e_i,\frac{x_i}{e_i}\right\}$.

Let $0<\varepsilon'\leq\frac{1}{2}$. Notice that
$$
\left(\frac{x^{1-\varepsilon}}{e_i},\frac{x}{e_i}\right]\subseteq\bigcup_{j=0}^{\ell} \left(\frac{x^{(1-\varepsilon')^{j+1}}}{e_i},\frac{x^{(1-\varepsilon')^{j}}}{e_i}\right],
$$
provided $(1-\varepsilon')^{\ell+1}\le 1-\varepsilon$ (equivalently:  $\ell+1\ge\log_{1-\varepsilon'}(1-\varepsilon)$). So, by the definition of $g_{i,\varepsilon}$, we get
\begin{equation}\label{eq:nierownosc5}
g_{i,\varepsilon}(x)\le \sum_{j=0}^{\lfloor\log_{1-\varepsilon'}(1-\varepsilon)\rfloor} g_{i,\varepsilon'}(x^{(1-\varepsilon')^j}).
\end{equation}
Note also that by (\ref{eq:star})
\begin{equation}\label{eq:star2}
\log_{1-\varepsilon'}(1-\varepsilon)=\frac{-\ln(1-\varepsilon)}{-\ln(1-\varepsilon')}\le 2\frac{\varepsilon}{\varepsilon'}
\end{equation}
provided $\varepsilon,\varepsilon'\in (0,\nicefrac{1}{2}]$.

\smallskip\noindent\emph{Verification of Erd\H{o}s' criterion.}
Fix $\varepsilon\in (0,\frac{1}{2})$ and let $i_0$ be such that $\frac{1}{2^{i_0+1}}<\varepsilon\le \frac{1}{2^{i_0}}$.

Since $\varepsilon$ and $i_0$ are fixed, the finitely many $o(1)$ terms in~\eqref{eq:nierownosc2} tend to zero as $x\to\infty$. Hence, by \eqref{eq:estim1} and \eqref{eq:nierownosc2},
\begin{equation}\label{eq:final1}
\limsup_{x\rightarrow\infty}\frac{1}{x}\left(\sum_{i=1}^{i_0}\sum_{p\in \mathscr{A}_i\cap (x^{1-\varepsilon}/e_i,x/e_i]}|[1,x]\cap e_ip\Z|\right)\le \sum_{i=1}^{i_0}\frac{1}{2^{i-1}}\varepsilon\le 2\varepsilon.
\end{equation}
Moreover, by \eqref{eq:estim1}, \eqref{eq:nierownosc5}, \eqref{eq:nierownosc3}\footnote{We recall that \eqref{eq:nierownosc3} holds for every $x$.} and \eqref{eq:star2} (for $\varepsilon'=\frac{1}{2^i}$)
\begin{multline}\label{eq:final2}
\sum_{i=i_0+1}^{\infty}\sum_{p\in \mathscr{A}_i\cap (x^{1-\varepsilon}/e_i,x/e_i]}|[1,x]\cap e_ip\Z|\le x\sum_{i=i_0+1}^{\infty}g_{i,\varepsilon}(x)\\
\leq x\sum_{i=i_0+1}^{\infty}\sum_{j=0}^{\lfloor\log_{1-\frac{1}{2^i}}(1-\varepsilon)\rfloor} g_{i,\frac{1}{2^i}}(x^{(1-\frac{1}{2^i})^j})\le  x\sum_{i=i_0+1}^{\infty}(\lfloor\log_{1-\frac{1}{2^i}}(1-\varepsilon)\rfloor+1) \frac{1}{4^{i-1}} \\
\leq x\sum_{i=i_0+1}^{\infty}2^{i+2}\varepsilon\frac{1}{4^{i-1}}=x\sum_{i=i_0+1}^{\infty}\frac{8\varepsilon}{2^{i-1}}\le 8x\varepsilon
\end{multline}
(to justify the 4th inequality we use also the fact that $2^{i+1}\varepsilon\ge 1$ for $i\ge i_0$).
It follows by \eqref{eq:bound2}, \eqref{eq:final1} and \eqref{eq:final2} that
\[
\limsup_{x\to\infty}\frac{1}{x}\sum_{a\in \mathscr{B} \cap(x^{1-\varepsilon}, x]}|[1,x]\cap a\Z\cap \mathcal{F}_{\mathscr{B}\cap [1,a)}|\le 2\varepsilon+8\varepsilon=10\varepsilon.
\]
Therefore,

\[
\lim\limits_{\varepsilon\rightarrow 0}\limsup_{x\to\infty}\frac{1}{x}\sum_{a\in \mathscr{B} \cap(x^{1-\varepsilon}, x]}|[1,x]\cap a\Z\cap \mathcal{F}_{\mathscr{B}\cap [1,a)}|=0.
\]
We conclude, using Theorem~\ref{besicovitch}, that $\mathscr{B}$ is indeed Besicovitch.

\smallskip\noindent\emph{Primitivity.}
Finally, assume that $\mathscr{E}$ is primitive. Since $K_i>e_i$ for $i\geq1$, it follows that $\mathscr{B}$ is primitive. Indeed, suppose that $e_ip\mid e_j q$ for some $p\in\mathscr{A}_i$ and $q\in\mathscr{A}_j$. If $i=j$, then $p=q$ since $p,q$ are primes. So $e_ip=e_j q$. If $i\neq j$ then $p\mid e_j$ because $\mathscr{A}_i$ and $\mathscr{A}_j$ are disjoint. Since
$p\geq K_i>e_i$, we have $e_i<e_j$. Notice that $e_i\nmid e_j$, as $\mathscr{E}$ is primitive. So $q\mid e_i$. Hence $q\leq e_i$ but $q\geq K_j>e_j>e_i$. We obtain a contradiction.
\end{proof}

\begin{corollary}\label{n1}
Under the assumptions of Theorem~\ref{jjj}, if we additionally assume that $\mathscr{A}_i$ are Behrend sets (e.g., $\mathscr{A}_i=\mathcal{P}_i\cap [K_i,\infty)$) and $\mathscr{E}$ is taut, then $\mathscr{B}'=\mathscr{E}$.
\end{corollary}
\begin{proof}
\begin{sloppypar}
This is a straightforward consequence of Theorem~\ref{pr:9_cor}.
\end{sloppypar}
\end{proof}

\begin{corollary}\label{n2}
Under the assumptions of Theorem~\ref{jjj}, if we additionally assume that each $\mathscr{A}_i$ is infinite, pairwise coprime and thin and $\mathscr{E}$ is minimal, then $\mathscr{B}^*=\mathscr{E}$. Moreover, still assuming that $\mathscr{E}$ is minimal, if $\bigcup_{i\geq 1}\mathscr{A}_i$ is thin, then $\mathscr{B}$ is taut. Such choices of the sets $\mathscr{A}_i$ can be obtained, for example, by taking sufficiently sparse subsets of $\mathcal{P}_i\cap [K_i,\infty)$.
\end{corollary}
\begin{proof}
If each $\mathscr{A}_i$ is infinite and pairwise coprime, then Theorem~\ref{wersja2_cor1} gives $\mathscr{B}^*=\mathscr{E}$. For the additional assertion, recall that $\mathscr{E}$ is minimal and hence primitive. Therefore, by the last assertion of Theorem~\ref{jjj}, the sequence $(K_i)$ can be chosen so that $\mathscr{B}$ is primitive. If $\bigcup_{i\geq1}\mathscr{A}_i$ is thin, then
\[
\sum_{b\in\mathscr{B}}\frac1b\leq \sum_{i\geq1}\sum_{a\in\mathscr{A}_i}\frac1a<\infty,
\]
so $\mathscr{B}$ is thin. Hence $\mathscr{B}$ is taut by \cite[implication~(7)]{MR3803141}.
\end{proof}

\subsection{The realization theorems}
We now derive the two realization theorems stated in the introduction.

\begin{proof}[Proof of Theorem~\ref{taut_bes}]
If $\mathscr{C}=\emptyset$, take $\mathscr{B}=\emptyset$. Otherwise, write $\mathscr{C}=\bigcup_{i\geq1}\{e_i\}$, repeating elements if $\mathscr{C}$ is finite. Choose $(K_i)_{i\geq1}$ as in Theorem~\ref{jjj} and take $\mathscr{A}_i=\mathcal{P}_i\cap[K_i,\infty)$. These sets are Behrend by~\eqref{eq:Prachar} and \cite[formula~(0.69)]{MR1414678}. Theorem~\ref{jjj} makes $\mathscr{B}=\bigcup_{i\geq1}e_i\mathscr{A}_i$ Besicovitch, and Corollary~\ref{n1} gives $\mathscr{B}'=\mathscr{C}$.
\end{proof}
\begin{proof}[Proof of Theorem~\ref{toeplitz_bes}]
If $\mathscr{D}=\emptyset$, take $\mathscr{B}=\emptyset$. Otherwise, write $\mathscr{D}=\bigcup_{i\geq1}\{e_i\}$, repeating elements if $\mathscr{D}$ is finite. Since $\mathscr{D}$ is primitive, choose $(K_i)_{i\geq1}$ as in the last assertion of Theorem~\ref{jjj}. For each $i$, choose an infinite subset $\mathscr{A}_i\subseteq\mathcal{P}_i\cap[K_i,\infty)$ such that $\sum_{p\in\mathscr{A}_i}1/p<2^{-i}$. Then $\bigcup_{i\geq1}\mathscr{A}_i$ is thin, and each $\mathscr{A}_i$ is infinite and pairwise coprime. By Theorem~\ref{jjj} and Corollary~\ref{n2}, the set $\mathscr{B}=\bigcup_{i\geq1}e_i\mathscr{A}_i$ is Besicovitch and taut and satisfies $\mathscr{B}^*=\mathscr{D}$.
\end{proof}
\section{Realizable triples}\label{triples}
In this section we will show that all triples $ijk\in\{0,1\}^3$ such that $ij\neq 01$ are realizable. In Section~\ref{podstawowe}, we recall the results mentioned in the introduction and we provide examples covering the remaining cases. Section~\ref{advanced} is devoted to delivering relevant examples with $\mathscr{B},\mathscr{B}',\mathscr{B}^*$ that are pairwise distinct.

\subsection{Basic examples}\label{podstawowe}
As noted in the introduction, we have the following:
\begin{itemize}
\item
Besicovitch's construction yields $ijk=001$ and $\mathscr{B}^*=\{1\}$. Its tautification is still non-Besicovitch by~\eqref{onlyobs}, so taking this tautification as $\mathscr{B}$ gives $ijk=001$ with $\mathscr{B}=\mathscr{B}'\neq\mathscr{B}^*=\{1\}$,
\item
Keller's example from~\cite{IrK} yields $ijk=000$. Indeed, since Keller's set is minimal, it is taut by \cite[Proposition~5.3]{MR3731019}, and hence $\mathscr{B}'=\mathscr{B}$. Moreover, $X_\eta$ is minimal, so its unique minimal subsystem is $X_\eta=X_{\eta^*}$. Since $\mathscr{B}$ and $\mathscr{B}^*$ are taut, Theorem~\ref{ThmL} gives $\mathscr{B}=\mathscr{B}^*$. Therefore $\mathscr{B}=\mathscr{B}'=\mathscr{B}^*$,
\item
any Erd\H{o}s set $\mathscr{B}$ yields $ijk=111$ with $\mathscr{B}=\mathscr{B}'\neq \mathscr{B}^*=\{1\}$.
\end{itemize}
\begin{example}[$ijk=101$]\label{ex101}
Consider Besicovitch's original example of a non-Besicovitch set $\mathscr{E}$. Now, it suffices to use Corollary~\ref{n1} (with $\mathscr{E}$ replaced by $\mathscr{E}'$) to find $\mathscr{B}$ that is Besicovitch and satisfies $\mathscr{B}'=\mathscr{E}'$. Since $\mathscr{B}^*=(\mathscr{B}')^*=(\mathscr{E}')^*=\mathscr{E}^*=\{1\}$ (see Section~\ref{prelim}), clearly $\mathscr{B},\mathscr{B}',\mathscr{B}^*$ are pairwise distinct.
\end{example}
\begin{example}[$ijk=100$]\label{b1}
Let $\mathscr{E}$ be a non-Besicovitch minimal set as in~\cite{IrK}. By Corollary~\ref{n1}, there exists a Besicovitch set $\mathscr{B}$ such that $\mathscr{B}'=\mathscr{E}$. It follows that we have a Besicovitch set $\mathscr{B}$ with $\mathscr{B}'=\mathscr{B}^*$ non-Besicovitch.
\end{example}
\begin{example}[$ijk=110$]\label{b2}
Let $\mathscr{E}$ be again a non-Besicovitch minimal set as in~\cite{IrK}. By Corollary~\ref{n2}, there exists a Besicovitch taut set $\mathscr{B}$ such that $\mathscr{B}^*=\mathscr{E}$. It follows that we have $\mathscr{B}=\mathscr{B}'$ Besicovitch with $\mathscr{B}^*$ non-Besicovitch.
\end{example}

\begin{remark}\label{dynamical_consequences_triples}
The examples above show that the Besicovitch property is not determined by the Mirsky measure or by the set of invariant measures. They also show that Besicovitch sets need not generate uniquely ergodic systems.

In each of the cases $001$, $101$ and $111$ considered above, we have $\mathscr{B}^*=\{1\}$, so that $\eta^*=\mathbf{0}$ is a regular Toeplitz sequence. Hence, by~\cite[Theorem~C]{DKPS},
\[
\mathscr{M}(X_\eta)=\mathscr{M}(X_{\eta'}).
\]
In particular, the two systems have the same ergodic invariant measures. Thus Example~\ref{ex101} shows that $\mathscr{B}$ and $\mathscr{B}'$ may have different natural-density behavior while the corresponding shifts have exactly the same invariant measures. Moreover, $\nu_\eta=\nu_{\eta'}$, but $\eta$ is generic for this common Mirsky measure whereas $\eta'$ is not, by the equivalence recalled in Section~\ref{intro}.

On the other hand, in the cases $000$, $100$ and $110$, the set $\mathscr{B}^*$ is minimal and non-Besicovitch. Consequently, $X_{\eta^*}$ cannot be uniquely ergodic: otherwise $\eta^*$ would be generic, which would force $d(\mathcal{M}_{\mathscr{B}^*})$ to exist. Hence neither $X_{\eta'}$ nor $X_\eta$ is uniquely ergodic. In particular, Example~\ref{b2} shows that even when both $\mathscr{B}$ and $\mathscr{B}'$ are Besicovitch, the corresponding $\mathscr{B}$-free system need not be uniquely ergodic.
\end{remark}

\subsection{Examples with pairwise distinct \texorpdfstring{$\mathscr{B}$, $\mathscr{B}'$ and $\mathscr{B}^*$}{B, B' and B*}}\label{advanced}
We will start this by an easy argument to show that $ijk=111$ can be realized with pairwise distinct $\mathscr{B},\mathscr{B}',\mathscr{B}^*$.

\begin{example}[$ijk=111$ with $\mathscr{B},\mathscr{B}',\mathscr{B}^*$ pairwise distinct]\label{4.1}
Take $\mathscr{B}=\mathscr{A}\cup 2\mathcal{P}$, where $\mathscr{A}\subseteq 2\N+1$ is an Erd\H{o}s set. By \cite[Theorem~1.6]{MR1414678}, $\mathscr{B}$ is Besicovitch as a union of two Besicovitch sets. Notice that $2\in C_\mathscr{B}$ and $\mathscr{A}\cup\{2\}$ is taut, where $C_\mathscr{B}$ is as in~\eqref{skaleC}. Therefore, by Proposition~\ref{pr:9}, we obtain $\mathscr{B}'=\mathscr{A}\cup \{2\}$. It follows that $\mathscr{B}'$ is an Erd\H{o}s set and thus is Besicovitch. Finally, it is not hard to see that we have $\mathscr{B}^*=\{1\}$ (indeed, $1\in D_\mathscr{B}$ as $\mathscr{B}$ contains an Erd\H{o}s set, $D_\mathscr{B}$ is as in \eqref{skaleD}).
\end{example}

\paragraph{A modified Besicovitch construction}
In order to provide the remaining examples, we will need a certain modification of the original Besicovitch's construction of non-Besicovitch sets. For an interval $I\subseteq\mathbb{R}$, we use the shorthand $\mathcal{M}_I:=\mathcal{M}_{I\cap\mathbb{N}}$. Let us begin by recalling the original scheme and point out what kind of changes we want to do. The main tool behind it is the following result of Erd\H{o}s~\cite{MR1574239}: $d_T:=d(\mathcal{M}_{(T,2T]})\to 0$ as $T\to\infty$, which implies that for any $T$, there exists $x_0=x_0(T)$ such that for all $x>x_0$, we have
\[
|\mathcal{M}_{(T,2T]}\cap [0,x]|\leq 2d_Tx.
\]
Fix $\vep>0$. We take the $T_k$ even because later we pass to the odd parts of these intervals; then $(T_k,2T_k]$ contains exactly $T_k/2$ odd integers. Let $T_k\in2\N$ tend to infinity fast enough, so that:
\begin{equation}\label{T_k_condition}
\sum_{k\geq 1}d_{T_k}\leq \varepsilon \text{ and }T_{k+1}>x_0(T_k) \text{ for all }k.
\end{equation}
Let $\mathscr{B}_{(T_k)}:=\bigcup_{k\geq 1}(T_k,2T_k]$. Then
\begin{equation}\label{12}
\overline{d}(\mathcal{M}_{\mathscr{B}_{(T_k)}})\geq \limsup_{k\to\infty} \frac{1}{2T_k}|\mathcal{M}_{\mathscr{B}_{(T_k)}} \cap [1,2T_k]|\geq \lim_{k\to\infty}\frac{1}{2T_k}|(T_k,2T_k]|=\frac{1}{2}.
\end{equation}
On the other hand,
\[
|\mathcal{M}_{\mathscr{B}_{(T_k)}}\cap [0,T_k]|\leq \sum_{h<k}|\mathcal{M}_{(T_h,2T_h]}\cap [0,T_k]| \leq \sum_{h<k}2d_{T_h}T_k\leq 2\varepsilon T_k,
\]
where the last two inequalities follow by the choice of $T_k$.
Thus, $\underline{d}(\mathcal{M}_{\mathscr{B}_{(T_k)}})\leq 2\varepsilon$.
Now, Besicovitch picks $\varepsilon <1/4$ and this is the end of the proof.

Notice that Besicovitch's example $\mathscr{B}_{(T_k)}$ is not taut because it is not primitive (any interval $(T_k,2T_k]$ contains some power of $2$). It is unknown whether $\mathscr{B}_{(T_k)}^{prim}$ is taut.
\begin{example}[$ijk=001$ with $\mathscr{B},\mathscr{B}',\mathscr{B}^*$ pairwise distinct]\label{EX1}
Now, consider
\[
\mathscr{B}_{\text{odd}}:=\mathscr{B}_{(T_k)}\cap (2\mathbb{Z}+1).
\]
 We claim that it is non-Besicovitch for $(T_k)$ chosen that satisfies \eqref{T_k_condition} for $\varepsilon<1/8$.
Indeed, clearly \begin{equation}\label{small_odd}
\underline{d}(\mathcal{M}_{\mathscr{B}_{\text{odd}}})\leq \underline{d}(\mathcal{M}_{\mathscr{B}_{(T_k)}})<2\varepsilon.
\end{equation} Now, since $\mathscr{B}_{\text{odd}}$ contains all odd numbers from $(T_k,2T_k]$, we obtain
\[
\frac{1}{2T_k}|\mathcal{M}_{\mathscr{B}_{\text{odd}}}\cap [1,2T_k]|\geq\frac{1}{2T_k}|(T_k,2T_k]\cap\mathscr{B}_{\text{odd}}|\geq
\frac{T_k}{4T_k}=\frac{1}{4}.
\]
Hence $\overline{d}(\mathcal{M}_{\mathscr{B}_{\text{odd}}})\geq 1/4$. So it suffices to take $\varepsilon <1/8$ to show that $\mathscr{B}_{\text{odd}}$ is non-Besicovitch.
By \eqref{density_difference} and \eqref{small_odd}, we have $1\notin(\mathscr{B}_{\text{odd}})'$. Since a prime can disappear under the primitive reduction defining the tautification only in the presence of $1$, every prime in $\mathscr{B}_{\text{odd}}$ belongs to $(\mathscr{B}_{\text{odd}})'$.
Let us consider
\[
\mathscr{B}:=2\mathcal{P}\cup (\mathscr{B}_{\text{odd}})'.
\]
It is clear that $\mathscr{B}$ is not taut as $2\mathcal{P}\subset \mathscr{B}$ is a rescaling of a Behrend set (see the tautness criterion \cite[Corollary~0.19]{MR1414678}). Since $\mathscr{B}_{\text{odd}}\subset2\Z+1$ and $(\mathscr{B}_{\text{odd}})'$ contains only factors of elements from $\mathscr{B}_{\text{odd}}$ (see the definition of tautification and \eqref{skaleC}), we have $(\mathscr{B}_{\text{odd}})'\subset2\Z+1$. Hence, $\{2\}\cup(\mathscr{B}_{\text{odd}})'$ is a primitive union of two taut sets, so in view of the closure property for primitive unions of taut sets recalled in Section~\ref{prelim} it is taut. It follows by Proposition \ref{pr:9} that
\begin{equation}\label{z:postac}
\mathscr{B}'=\{2\}\cup (\mathscr{B}_{\text{odd}})'.
\end{equation}
Since for any $T\geq2$, each set $(T,2T]$ contains an odd prime by Bertrand's postulate, it follows that $\mathscr{B}$ contains an infinite coprime subset. Thus, $\mathscr{B}^*=\{1\}$ as $1\in D_\mathscr{B}$ (recall~\eqref{skaleD}). Now, one can easily see that $\mathscr{B}$, $\mathscr{B}'$ and $\mathscr{B}^*$ are pairwise distinct (indeed, $2\not\in \mathscr{B}\cup\mathscr{B}^*$ while $2\in\mathscr{B}'$, and $1\in\mathscr{B}^*\setminus\mathscr{B}$).
It remains to show that both $\mathscr{B}$ and $\mathscr{B}'$ are non-Besicovitch. We will begin by showing it for $\mathscr{B}'$.
Notice that for any $N\geq1$ we have
\begin{equation}\label{B'splitted}
\frac{1}{N}|\mathcal{M}_{\mathscr{B}'}\cap [1,N]|=\frac{1}{N}|2\Z\cap[1,N]|+\frac{1}{N}|\mathcal{M}_{(\mathscr{B}_{\text{odd}})'}\cap [1,N]\cap(2\Z+1)|.
\end{equation}
Since $\mathcal{M}_{\mathscr{B}_{\text{odd}}}\subset\mathcal{M}_{(\mathscr{B}_{\text{odd}})'}$, by \eqref{B'splitted}, for any $k\geq1$ we get
\[
\frac{1}{2T_k}|\mathcal{M}_{\mathscr{B}'}\cap [1,2T_k]|\geq
\frac{1}{2}+\frac{1}{2T_k}|[1,2T_k]\cap\mathcal{M}_{\mathscr{B}_{\text{odd}}}\cap(2\Z+1)|.\]
\begin{sloppypar}
Now, since any odd number from $(T_k,2T_k]$ is contained in $\mathscr{B}_{\text{odd}}$ and $\mathcal{M}_{\mathscr{B}_{\text{odd}}}\subset\mathcal{M}_{(\mathscr{B}_{\text{odd}})'}$, we obtain
\end{sloppypar}
\[\frac{1}{2T_k}|\mathcal{M}_{\mathscr{B}'}\cap [1,2T_k]|\geq\frac{1}{2}+\frac{1}{2T_k}|(T_k,2T_k]\cap\mathcal{M}_{\mathscr{B}_{\text{odd}}}\cap(2\Z+1)|\geq
\frac{1}{2}+\frac{T_k}{4T_k}=\frac{3}{4}.\]
Finally, we have
\[\overline{d}(\mathcal{M}_{\mathscr{B}'})\geq\frac{3}{4}.\]

On the other hand,
by \eqref{B'splitted} we have \[\frac{1}{N}|\mathcal{M}_{\mathscr{B}'}\cap[1,N]|\leq \frac{1}{2}+\frac{1}{N}|\mathcal{M}_{(\mathscr{B}_{\text{odd}})'}\cap[1,N]|.\]
Hence \[\underline{d}(\mathcal{M}_{\mathscr{B}'})\leq\frac{1}{2}+\underline{d}(\mathcal{M}_{(\mathscr{B}_{\text{odd}})'})=\frac{1}{2}+\underline{d}(\mathcal{M}_{\mathscr{B}_{\text{odd}}}),\]
where the last equality follows from $\underline{d}(\mathcal{M}_{(\mathscr{B}_{\text{odd}})'})=\underline{d}(\mathcal{M}_{\mathscr{B}_{\text{odd}}})$, see \eqref{density_difference}.
By \eqref{small_odd}, we get \[\underline{d}(\mathcal{M}_{\mathscr{B}'})<\frac{1}{2}+2\varepsilon.\]
So it suffices to take $\varepsilon <1/8$, to conclude that $\mathscr{B}'$ is non-Besicovitch. Since $\mathcal{M}_\mathscr{B}$ and $\mathcal{M}_{\mathscr{B}'}$ differ by a finite set $\{\pm 2\}$ and $\mathscr{B}'$ is non-Besicovitch, the set $\mathscr{B}$ is non-Besicovitch.\hfill$\diamond$
\end{example}
\paragraph{Keller's construction and its odd part}
Recall that Keller in~\cite{IrK} modified Besicovitch's construction to deliver a non-Besicovitch minimal set $\mathscr{B}_{\text{Kel}}^{prim}\subseteq\mathscr{B}_{(T_k)}$ such that $\overline{d}(\mathcal{M}_{\mathscr{B}_{\text{Kel}}})>\frac{1}{2}-\varepsilon$ and $\underline{d}(\mathcal{M}_{\mathscr{B}_{\text{Kel}}})<\varepsilon$. He constructed $\mathscr{B}_{\text{Kel}}$ using an inductive procedure involving a sequence $(n_j)$ and finite sets $P_j\subset\mathcal{P}$, $j\geq1$ such that
\begin{enumerate}[(I)]
\item $\mathscr{B}_{\text{Kel}}^{prim}$ is minimal, where $\mathscr{B}_{\text{Kel}}=\bigcup_{j\geq1}[T_{n_j},2T_{n_j})\setminus\bigcup_{j\geq1}j\mathcal{F}_{P_j}^{(1)}$ and $\mathcal{F}_{P_j}^{(1)}=\mathcal{F}_{P_j}\setminus\{1\}$,\label{Keller1}
\item $\left|\bigcup_{i\geq1}i\mathcal{F}_{P_i}^{(1)}\cap\left[T_{n_j}, 2T_{n_j}\right)\right| \leqslant \varepsilon T_{n_j}$ for any $j\geq1$. \label{Keller2}
\end{enumerate}
\begin{remark}\label{keller_ex_taut}
\begin{sloppypar}
Notice that $\mathscr{B}_{\text{Kel}}^{prim}$ is taut since it is minimal, see \cite[Proposition~5.3]{MR3731019}.
\end{sloppypar}
\end{remark}
In the next three examples, take $0<\varepsilon<1/8$ and use
\begin{equation}\label{odd_Keller}
\mathscr{B}_{\text{odd}}=\mathscr{B}_{\text{Kel}}^{prim}\cap(2\Z+1).
\end{equation}
Put $\mathscr{R}=\bigcup_{i\geq1}i\mathcal{F}_{P_i}^{(1)}$. Every odd integer in $[T_{n_j},2T_{n_j})\setminus\mathscr{R}$ belongs to $\mathscr{B}_{\text{Kel}}\subseteq\mathcal{M}_{\mathscr{B}_{\text{Kel}}}=\mathcal{M}_{\mathscr{B}_{\text{Kel}}^{prim}}$ and hence to $\mathcal{M}_{\mathscr{B}_{\text{odd}}}$. By \eqref{Keller1} and \eqref{Keller2}, we obtain
\begin{equation}\label{B_odd}
\begin{aligned}
&\frac{1}{2T_{n_j}}|\mathcal{M}_{\mathscr{B}_{\text{odd}}}
  \cap[1,2T_{n_j}]\cap(2\Z+1)|\\
&\quad\geq \frac{1}{2T_{n_j}}|\mathcal{M}_{\mathscr{B}_{\text{odd}}}
  \cap[T_{n_j},2T_{n_j})\cap(2\Z+1)|\\
&\quad\geq\frac{\lceil T_{n_j}/2\rceil-\varepsilon T_{n_j}}{2T_{n_j}}
\geq\frac{1}{4}-\varepsilon.
\end{aligned}
\end{equation}

Since $\mathscr{B}_{\text{odd}}\subset\mathscr{B}_{\text{Kel}}^{prim}$ and $\underline{d}(\mathcal{M}_{\mathscr{B}_{\text{Kel}}})<\varepsilon$, by~\eqref{B_odd}, we have
\[\overline{d}(\mathcal{M}_{\mathscr{B}_{\text{odd}}})\geq\frac{1}{4}-\varepsilon\text{ and }\underline{d}(\mathcal{M}_{\mathscr{B}_{\text{odd}}})<\varepsilon.\] So $\mathscr{B}_{\text{odd}}$ is a non-Besicovitch minimal set (it is a subset of the minimal set $\mathscr{B}_{\text{Kel}}^{prim}$; see Section~\ref{prelim}, \emph{Minimal sets}).

Set $\mathscr{D}:=\mathscr{B}_{\text{odd}}\cup\{2\}$. We will show that $\mathscr{D}$ is non-Besicovitch. By the corresponding splitting for $\mathcal{M}_{\mathscr{D}}$ as in \eqref{B'splitted}, for any $j\geq1$ we get
\[
\frac{1}{2T_{n_j}}|\mathcal{M}_{\mathscr{D}}\cap [1,2T_{n_j}]|\geq
\frac{1}{2}+\frac{1}{2T_{n_j}}|[1,2T_{n_j}]\cap\mathcal{M}_{\mathscr{B}_{\text{odd}}}\cap(2\Z+1)|.\]

Hence $\overline{d}(\mathcal{M}_{\mathscr{D}})=\frac{1}{2}+\overline{d}(\mathcal{M}_{\mathscr{B}_{\text{odd}}}\cap(2\Z+1))\geq\frac{3}{4}-\varepsilon$. On the other hand, again by the corresponding splitting for $\mathcal{M}_{\mathscr{D}}$ as in \eqref{B'splitted}, we have \[
\frac{1}{N}|\mathcal{M}_{\mathscr{D}}\cap[1,N]|\leq \frac{1}{2}+\frac{1}{N}|\mathcal{M}_{\mathscr{B}_{\text{odd}}}\cap[1,N]|.\]
Hence
\[\underline{d}(\mathcal{M}_{\mathscr{D}})\leq\frac{1}{2}+\underline{d}(\mathcal{M}_{\mathscr{B}_{\text{odd}}})\leq\frac{1}{2}+\underline{d}(\mathcal{M}_{\mathscr{B}_{\text{Kel}}}).\]
So $\underline{d}(\mathcal{M}_{\mathscr{D}})<\frac{1}{2}+\varepsilon$. If $\varepsilon<1/8$ then $\mathscr{D}$ is non-Besicovitch.

\begin{example}[$ijk=100$ with $\mathscr{B},\mathscr{B}',\mathscr{B}^*$ pairwise distinct]\label{EX2}
Using $\mathscr{B}_{\text{odd}}$ from~\eqref{odd_Keller}, let
\[
\mathscr{C}=\mathscr{B}_{\text{odd}}\cup\{2p^2: p\in\mathcal{P}\}.
\]
Since all elements of $\mathscr{B}_{\text{odd}}$ are odd, we have $2p^2\ndivides b$ for any $b\in\mathscr{B}_{\text{odd}}$ and any $p\in\mathcal{P}$. So \[\mathscr{B}_{\text{odd}}\subset\mathscr{C}^{prim}.\]
Let $Q:=\{p\in\mathcal{P}:2p^2\in\mathscr{C}^{prim}\}$. We first note that $Q$ is infinite. Otherwise, for all but finitely many primes $p$, the element $2p^2$ would be divisible by an element of $\mathscr{B}_{\text{odd}}$; since $1\notin\mathscr{B}_{\text{odd}}$, this element is $p$ or $p^2$. This would give an infinite pairwise coprime subset of the minimal set $\mathscr{B}_{\text{odd}}$, a contradiction to \cite[Theorem~B]{MR3947636}. Therefore $2\{p^2:p\in Q\}\subseteq\mathscr{C}^{prim}$, so $\mathscr{C}^{prim}$ is not minimal, again by \cite[Theorem~B]{MR3947636}. Now we will show that $\mathscr{C}^{prim}$ is taut. Indeed, $\mathscr{C}^{prim}\setminus\mathscr{B}_{\text{odd}}$ is taut as a subset of a taut set $\{2p^2: p\in\mathcal{P}\}$ (see \cite[implication~(7)]{MR3803141}).
Moreover, $\mathscr{B}_{\text{odd}}$ is taut as a minimal set (\cite[Proposition~5.3]{MR3731019}). By the primitivity of $\mathscr{C}^{prim}=(\mathscr{C}^{prim}\setminus\mathscr{B}_{\text{odd}})\cup\mathscr{B}_{\text{odd}}$ and the closure property for primitive unions of taut sets recalled in Section~\ref{prelim}, we conclude that $\mathscr{C}^{prim}$ is indeed taut.
We claim
\[
\mathscr{C}^*=\mathscr{B}_{\text{odd}}\cup\{2\}.
\]
Since $\mathscr{B}_{\text{odd}}$ and $\{2\}$ are minimal and $\mathscr{B}_{\text{odd}}\subset2\Z+1$, $\mathscr{B}_{\text{odd}}\cup\{2\}$ is a primitive union of two minimal sets. Hence by the closure property for primitive unions of minimal sets recalled in Section~\ref{prelim}, $\mathscr{B}_{\text{odd}}\cup\{2\}$ is minimal. Since $Q$ is infinite, $2\in D_{\mathscr{C}^{prim}}$. Moreover, $(\mathscr{C}^{prim}\cup\{2\})^{prim}=\mathscr{B}_{\text{odd}}\cup\{2\}$, which is minimal. Hence Proposition~\ref{wersja2} gives $(\mathscr{C}^{prim})^*=\mathscr{B}_{\text{odd}}\cup\{2\}$, and therefore $\mathscr{C}^*=\mathscr{B}_{\text{odd}}\cup\{2\}$. Thus $\mathscr{C}^*$ is non-Besicovitch by the common preparation above.
By Theorem \ref{taut_bes}, there exists a Besicovitch set $\mathscr{B}$ such that $\mathscr{B}'=\mathscr{C}^{prim}$. Notice that $\mathscr{C}^*=(\mathscr{C}^{prim})^*=(\mathscr{B}')^*=\mathscr{B}^*$. Since $\mathscr{B}^*$ is minimal and $\mathscr{C}^{prim}$ is not minimal, we have $\mathscr{B}^*\neq\mathscr{C}^{prim}$.
 So it remains to show that $\mathscr{C}$ is non-Besicovitch (it is equivalent that $\mathscr{C}^{prim}$ is non-Besicovitch because $\mathscr{C}$ and $\mathscr{C}^{prim}$ have the same sets of multiples). Notice that for any $N\geq1$ we have
\begin{equation}\label{Csplitted}
\frac{1}{N}|\mathcal{M}_{\mathscr{C}}\cap [1,N]|\geq\frac{1}{N}|\mathcal{M}_{\{2p^2: p\in\mathcal{P}\}}\cap[1,N]|+\frac{1}{N}|\mathcal{M}_{\mathscr{B}_{\text{odd}}}\cap [1,N]\cap(2\Z+1)|.
\end{equation}
Since $d(\mathcal{F}_{\{p^2\ :\ p\in\mathcal{P} \}})=\frac{6}{\pi^2}$, for sufficiently large $k\geq1$ we have
\begin{align*}
&\frac{1}{2T_{n_k}}|\mathcal{M}_\mathscr{C} \cap [1,2T_{n_k}]|\\
&\qquad\geq \frac{1}{2}\left(1-\frac{6}{\pi^2}\right)-\varepsilon
+\frac{1}{2T_{n_k}}|\mathcal{M}_{\mathscr{B}_{\text{odd}}}\cap [T_{n_k},2T_{n_k})\cap (2\Z+1)|\\
&\qquad\geq\frac{1}{2}-\frac{3}{\pi^2}-\varepsilon+\frac{T_{n_k}}{4T_{n_k}}-\varepsilon
=\frac{3}{4}-\frac{3}{\pi^2}-2\varepsilon,
\end{align*}
where the last inequality follows from \eqref{B_odd}.
Hence \[\overline{d}(\mathcal{M}_{\mathscr{C}})\geq\frac{3}{4}-\frac{3}{\pi^2}-2\varepsilon.\]
On the other hand, for sufficiently large $N\geq1$ we get
\begin{align*}
&\frac{1}{N}|\mathcal{M}_{\mathscr{C}}\cap [1,N]|\leq \frac{1}{N}|[1,N]\cap\mathcal{M}_{\{2p^2:p\in\mathcal{P}\}}|+\frac{1}{N}|[1,N]\cap\mathcal{M}_{\mathscr{B}_{\text{odd}}}|\\
&\leq\frac{1}{2}-\frac{3}{\pi^2}+\varepsilon+\frac{1}{N}|[1,N]\cap\mathcal{M}_{\mathscr{B}_{\text{odd}}}|.
\end{align*}
So
\begin{align*}
&\underline{d}(\mathcal{M}_\mathscr{C})\leq\frac{1}{2}-\frac{3}{\pi^2}+\varepsilon+\underline{d}(\mathcal{M}_{\mathscr{B}_{\text{odd}}})\leq\frac{1}{2}-\frac{3}{\pi^2}+\varepsilon+\underline{d}(\mathcal{M}_{\mathscr{B}_{\text{Kel}}^{prim}})\\
&=\frac{1}{2}-\frac{3}{\pi^2}+\varepsilon+\underline{d}(\mathcal{M}_{\mathscr{B}_{\text{Kel}}})<\frac{1}{2}-\frac{3}{\pi^2}+2\varepsilon,
\end{align*}
\begin{sloppypar}
where the second inequality follows from $\mathscr{B}_\text{odd}\subset\mathscr{B}_{\text{Kel}}^{prim}\subset\mathscr{B}_{\text{Kel}}$
and $\mathcal{M}_{\mathscr{B}_{\text{Kel}}^{prim}}=\mathcal{M}_{\mathscr{B}_{\text{Kel}}}$.
So it is enough to take $\varepsilon<1/16$ to show that $\mathscr{C}$ is non-Besicovitch.
The example satisfies all required conditions.\hfill$\diamond$
\end{sloppypar}
\end{example}
\begin{example}[$ijk=110$ with $\mathscr{B},\mathscr{B}',\mathscr{B}^*$ pairwise distinct] \label{110}
Use the non-Besicovitch minimal set $\mathscr{B}_{\text{odd}}$ from~\eqref{odd_Keller}.

By Theorem~\ref{toeplitz_bes}, there exists a Besicovitch taut set $\mathscr{C}$ such that $\mathscr{C}^*=\mathscr{B}_{\text{odd}}$. The main tool of the proof of Theorem~\ref{toeplitz_bes} is Theorem~\ref{jjj}. The construction used in the proof of the latter one yields $\mathscr{C}$ of the form $\bigcup_{b\in\mathscr{B}_{\text{odd}}}b\mathscr{A}_b$ for some infinite $\mathscr{A}_b\subset (2^{i_b+1}\Z+2^{i_b}+1)\cap\mathcal{P}$, $b\in\mathscr{B}_{\text{odd}}$. So we have $\mathscr{C}\subset2\Z+1$ and any element of $\mathscr{C}$ has at least two prime factors (not necessarily distinct). Consider \[\mathscr{B}=\mathscr{C}\cup2\mathcal{P}.\] By above, we obtain that $\mathscr{B}$ is primitive. Since $\mathscr{C}$ and $2\mathcal{P}$ are Besicovitch, by \cite[Theorem~1.6]{MR1414678}, $\mathscr{B}$ is Besicovitch too. We claim that
\[\mathscr{B}'=\mathscr{C}\cup\{2\}\text{ and }\mathscr{B}^*=\mathscr{C}^*\cup\{2\}.\]Since $\mathscr{C}$ and $\{2\}$ are  Besicovitch and taut and $\mathscr{C}\subset2\Z+1$, also $\mathscr{C}\cup\{2\}$ is Besicovitch and taut by \cite[Theorem~1.6]{MR1414678} and the closure property for primitive unions of taut sets recalled in Section~\ref{prelim}. Since $\mathcal{P}$ is Behrend and $2\mathcal{P}\subseteq\mathscr{B}$, we have $2\in C_\mathscr{B}$, and $(\mathscr{B}\cup\{2\})^{prim}=\mathscr{C}\cup\{2\}$. Hence by Proposition~\ref{pr:9}, we obtain $\mathscr{B}'=\mathscr{C}\cup\{2\}$. Since each $\mathscr{A}_b$ and $\mathcal{P}$ is an infinite pairwise coprime set, $\mathscr{D}:=\mathscr{B}_{\text{odd}}\cup\{2\}\subseteq D_\mathscr{B}$. Moreover, $(\mathscr{B}\cup\mathscr{D})^{prim}=\mathscr{D}$, and $\mathscr{D}$ is minimal. Hence Proposition~\ref{wersja2} gives $\mathscr{B}^*=\mathscr{D}=\mathscr{C}^*\cup\{2\}$.
The common preparation above shows that $\mathscr{B}_{\text{odd}}\cup\{2\}$ is non-Besicovitch. Finally, $\mathscr{B}$, $\mathscr{B}'$ and $\mathscr{B}^*$ are pairwise distinct: $2$ belongs to $\mathscr{B}'$ and to $\mathscr{B}^*$, but not to $\mathscr{B}$, while $\mathscr{B}'\neq\mathscr{B}^*$ because $\mathscr{C}\neq\mathscr{B}_{\text{odd}}$, the former being Besicovitch and the latter not. The assertion follows.\hfill$\diamond$
\end{example}
\begin{example}[$ijk=000$ with $\mathscr{B},\mathscr{B}',\mathscr{B}^*$ pairwise distinct]\label{000}
Let $P_p\subset\mathcal{P}$, $p\in\mathcal{P}$, be disjoint sets such that $\sum_{q\in P_p}\frac{1}{q}$ is divergent for any $p\in\mathcal{P}$ and use $\mathscr{B}_{\text{odd}}$ from~\eqref{odd_Keller}. Let us consider \[\mathscr{B}=\mathscr{B}_{\text{odd}}\cup\bigcup_{p\in\mathcal{P}}2p^2P_p.\] Then $\mathscr{B}$ is obviously not taut as any $P_p$ is Behrend, see \cite[formula~(0.69)]{MR1414678} and \cite[Corollary~0.19]{MR1414678}. We claim that \[\mathscr{B}'=(\mathscr{B}_{\text{odd}}\cup\{2p^2 : \ p\in\mathcal{P}\})^{prim} \]
and \[\mathscr{B}^*=\mathscr{B}_{\text{odd}}\cup\{2\}.\]
Indeed, by the closure property for primitive unions of taut sets recalled in Section~\ref{prelim}, $(\mathscr{B}_{\text{odd}}\cup\{2p^2 : \ p\in\mathcal{P}\})^{prim}$ is taut because any subset of $\{2p^2 : \ p\in\mathcal{P}\}$ and $\mathscr{B}_{\text{odd}}$ are taut.
Since each $P_p$ is Behrend and $2p^2P_p\subseteq\mathscr{B}$, we have $\{2p^2:p\in\mathcal{P}\}\subseteq C_\mathscr{B}$. Moreover, we get $(\mathscr{B}\cup\{2p^2:p\in\mathcal{P}\})^{prim}=(\mathscr{B}_{\text{odd}}\cup\{2p^2:p\in\mathcal{P}\})^{prim}$.
So by Proposition~\ref{pr:9}, we get $\mathscr{B}'=(\mathscr{B}_{\text{odd}}\cup\{2p^2 : \ p\in\mathcal{P}\})^{prim}$. Since $\mathscr{B}_{\text{odd}}\subset2\Z+1$ and $\{2\}$ are minimal and $\mathscr{B}_{\text{odd}}\cup\{2\}$ is primitive, by the closure property for primitive unions of minimal sets recalled in Section~\ref{prelim}, $\mathscr{B}_{\text{odd}}\cup\{2\}$ is minimal too.
Moreover,
\[
\mathscr{B}^*=(\mathscr{B}')^*=((\mathscr{B}_{\text{odd}}\cup\{2p^2 : \ p\in\mathcal{P}\})^{prim})^*=(\mathscr{B}_{\text{odd}}\cup\{2p^2 : \ p\in\mathcal{P}\})^*
\]
By Example~\ref{EX2}, we have
\[
\mathscr{B}^*=\mathscr{B}_{\text{odd}}\cup\{2\}.
\]

The common preparation above shows that $\mathscr{B}^*$ is non-Besicovitch. Let
\[
\mathscr{C}:=\mathscr{B}_{\text{odd}}\cup\{2p^2:p\in\mathcal{P}\}.
\]
Since $\mathscr{B}'=\mathscr{C}^{prim}$ and primitive reduction does not change the set of multiples, Example~\ref{EX2} shows that $\mathscr{B}'$ is non-Besicovitch. Moreover, by \eqref{density_difference}, the estimates in Example~\ref{EX2}, and \eqref{B_odd}, taking $0<\varepsilon<1/128$, we obtain
\[
\underline{d}(\mathcal{M}_{\mathscr{B}})
=\underline{d}(\mathcal{M}_{\mathscr{B}'})
=\underline{d}(\mathcal{M}_{\mathscr{C}})
<\frac{1}{2}-\frac{3}{\pi^2}+2\varepsilon
<\frac{1}{4}-\varepsilon
\leq\overline{d}(\mathcal{M}_{\mathscr{B}}).
\]
Hence $\mathscr{B}$ is non-Besicovitch as well. This completes the proof.\hfill$\diamond$
\end{example}
\section{Density of \texorpdfstring{$\mathcal{M}_{\mathscr{B}'}\setminus\mathcal{M}_{\mathscr{B}}$}{M(B') minus M(B)}}\label{easy}

This section addresses Question~\ref{q2}. We first treat non-Besicovitch tautification: Proposition~\ref{bazujenaDE} refines Corollary~\ref{prop2}, and Proposition~\ref{prop:nonbes-choice} shows that $\mathscr{B}$ may also be chosen to be non-Besicovitch. We then construct a primitive thin set of positive upper Banach density and use the Besicovitch loosening from Section~\ref{se4} to prove Theorem~\ref{prop1}.

\subsection{Non-Besicovitch tautification}

\begin{proposition}\label{bazujenaDE}
Let $\mathscr{E}$ be non-Besicovitch. Then there exists a strictly increasing sequence $(K_i)_{i\geq 1}$ such that, setting $\mathscr{E}_i=\mathscr{E}\cap [1,K_i)$, for every choice of $\emptyset\neq\mathscr{A}_i\subset [K_i,\infty)$ and
\[
\mathscr{B}=\bigcup_{i\geq 1}\mathscr{E}_i\mathscr{A}_i,
\]
we have $\overline{d}(\mathcal{M}_\mathscr{E}\setminus \mathcal{M}_\mathscr{B})>0$. If we additionally assume that $\mathscr{E}$ is taut and that $\mathscr{A}_i$, $i\geq1$, are Behrend sets, then $\mathscr{B}'=\mathscr{E}$.
\end{proposition}

\begin{proof}[Proof of Proposition~\ref{bazujenaDE}]
If $\mathscr{E}$ is not Besicovitch,
\begin{equation}\label{non-Bes}
\begin{aligned}
&\text{there exists }\vep_0>0 \text{ such that for all finite subsets }S\subseteq\mathscr{E},\\
&\text{ we have } \overline{d}(\cM_\mathscr{E}\setminus\cM_S)>\vep_0
\end{aligned}
\end{equation}
(this follows from the Davenport-Erd\H{o}s theorem, ad absurdum). Let $K_1:=2$ and $\mathscr{E}_1:=\mathscr{E}\cap [1,K_1)$. Suppose that we have defined $K_1,\dots, K_i\in \N$ and $\mathscr{E}_1\subseteq \dots\subseteq \mathscr{E}_i$. Using~\eqref{non-Bes} for $S=\mathscr{E}_i$, we can find $K_{i+1}>K_i$ such that
\begin{equation*}
|[1,K_{i+1})\cap(\cM_\mathscr{E}\setminus\cM_{\mathscr{E}_i})|\geq K_{i+1}\frac{\vep_0}{2}.
\end{equation*}
Let $\mathscr{E}_{i+1}:=\mathscr{E}\cap [1,K_{i+1})$. Finally, let
\[
\mathscr{B}:=\bigcup_{i\geq 1}\mathscr{B}_i, \text{ where }\mathscr{B}_i=\mathscr{E}_i\mathscr{A}_i,
\]
where $\emptyset \neq \mathscr{A}_i\subset [K_i,\infty)$. It follows that
\begin{align*}
&|[1,K_{i+1})\cap(\mathcal{M}_\mathscr{E}\setminus \mathcal{M}_\mathscr{B})|=|[1,K_{i+1})\cap(\mathcal{M}_{\mathscr{E}}\setminus \mathcal{M}_{\mathscr{B}_1\cup\dots\cup \mathscr{B}_i})|
\\
&\geq |[1,K_{i+1})\cap(\mathcal{M}_{\mathscr{E}}\setminus \mathcal{M}_{\mathscr{E}_i})|\geq K_{i+1} \frac{\vep_0}{2}.
\end{align*}
Thus, $\overline{d}(\mathcal{M}_\mathscr{E}\setminus \mathcal{M}_\mathscr{B})\geq \frac{\vep_0}{2}$. If $\mathscr{E}$ is taut and $\mathscr{A}_i$ are Behrend sets, it remains to use Theorem~\ref{pr:9_cor} to conclude that $\mathscr{B}'=\mathscr{E}$.
\end{proof}

\begin{proposition}\label{prop:nonbes-choice}
Let $\mathscr{E}$ be a non-Besicovitch taut set. Then there exists a set $\mathscr{B}$ such that $\mathscr{B}'=\mathscr{E}$,
\[
\overline{d}(\mathcal{M}_{\mathscr{E}}\setminus\mathcal{M}_{\mathscr{B}})>0,
\]
and
\begin{equation}\label{eq:matching-densities}
\underline{d}(\mathcal{M}_{\mathscr{B}})=\underline{d}(\mathcal{M}_{\mathscr{E}})
\quad\text{and}\quad
\overline{d}(\mathcal{M}_{\mathscr{B}})=\overline{d}(\mathcal{M}_{\mathscr{E}}).
\end{equation}
In particular, $\mathscr{B}$ is non-Besicovitch.
\end{proposition}
\begin{proof}
Put
\[
\delta:=\underline{d}(\mathcal{M}_{\mathscr{E}}),
\qquad
\alpha:=\overline{d}(\mathcal{M}_{\mathscr{E}}).
\]
Since $\mathscr{E}$ is non-Besicovitch, $\delta<\alpha$. Let $\vep_0>0$ be as in~\eqref{non-Bes}, and let $(\gamma_i)_{i\geq1}$ be a sequence of positive numbers decreasing to zero. We construct two increasing sequences $(G_i)_{i\geq1}$ and $(H_i)_{i\geq0}$, with $H_0:=0$ and
\[
H_{i-1}<G_i<H_i,
\]
as follows. Suppose that $H_{i-1}$ has already been chosen and set
\[
S_i:=\mathscr{E}\cap[1,H_{i-1}].
\]
By~\eqref{non-Bes}, there exist arbitrarily large $G_i$ such that
\begin{equation}\label{eq:gap-points}
\frac{1}{G_i}\big|[1,G_i]\cap(\mathcal{M}_{\mathscr{E}}\setminus\mathcal{M}_{S_i})\big|>\frac{\vep_0}{2}.
\end{equation}
After choosing such a $G_i>H_{i-1}$, choose $H_i>G_i$ so large that
\begin{equation}\label{eq:recovery-points}
\frac{1}{H_i}\big|[1,H_i]\cap\mathcal{M}_{\mathscr{E}}\big|>\alpha-\gamma_i
\qquad\text{and}\qquad
\frac{G_i}{H_i}<\gamma_i.
\end{equation}
This is possible because the first quotient has limsup equal to $\alpha$.

For each $i\geq1$, let
\[
\mathscr{R}_i:=\mathcal{M}_{\mathscr{E}}\cap(G_i,H_i].
\]
Furthermore, for every $e\in\mathscr{E}\cap(H_{i-1},H_i]$, choose a Behrend set $\mathscr{A}_e$ so that
\begin{equation}\label{eq:far-out-witness}
e\mathscr{A}_e\subset(H_i,\infty).
\end{equation}
Such a choice is always possible: take any Behrend set and then a sufficiently far tail, which remains Behrend by \cite[Corollary~0.14]{MR1414678}. Define
\begin{equation}\label{eq:recovery-construction}
\mathscr{B}:=\bigcup_{i\geq1}\mathscr{R}_i\ \cup\ \bigcup_{e\in\mathscr{E}}e\mathscr{A}_e.
\end{equation}
Every element of $\mathscr{B}$ is a multiple of an element of $\mathscr{E}$, and hence
\begin{equation}\label{eq:MB-subset-ME}
\mathcal{M}_{\mathscr{B}}\subseteq\mathcal{M}_{\mathscr{E}}.
\end{equation}

We first verify the positive-density gap. We claim that
\begin{equation}\label{eq:gap-inclusion}
\mathcal{M}_{\mathscr{B}}\cap[1,G_i]\subseteq\mathcal{M}_{S_i}\cap[1,G_i]
\qquad(i\geq1).
\end{equation}
Indeed, let $n\leq G_i$ be divisible by some $b\in\mathscr{B}$. If $b\in\mathscr{R}_j$ with $j\geq i$, then $b>G_j\geq G_i$, which is impossible. If $b\in\mathscr{R}_j$ with $j<i$, then $b\leq H_j\leq H_{i-1}$ and, since $b\in\mathcal{M}_{\mathscr{E}}$, we obtain that $b$ is divisible by some $e\in S_i$. Hence, we have $n\in\mathcal{M}_{S_i}$. Similarly, if $b=e a$ with $a\in\mathscr{A}_e$ and $e\in\mathscr{E}\cap(H_{j-1},H_j]$, then for $j\geq i$ we have $b>H_j>G_i$ by~\eqref{eq:far-out-witness}, while for $j<i$ we have $e\leq H_{i-1}$ and hence again $n\in\mathcal{M}_{S_i}$. This proves~\eqref{eq:gap-inclusion}. Combining it with~\eqref{eq:gap-points}, we obtain
\[
\frac{1}{G_i}\big|[1,G_i]\cap(\mathcal{M}_{\mathscr{E}}\setminus\mathcal{M}_{\mathscr{B}})\big|>\frac{\vep_0}{2},
\]
and therefore
\begin{equation}\label{eq:positive-gap-nonbes-choice}
\overline{d}(\mathcal{M}_{\mathscr{E}}\setminus\mathcal{M}_{\mathscr{B}})\geq\frac{\vep_0}{2}>0.
\end{equation}

Next we compute the upper density of $\mathcal{M}_{\mathscr{B}}$. Since $\mathscr{R}_i\subseteq\mathscr{B}$, we have
\[
\mathcal{M}_{\mathscr{E}}\cap(G_i,H_i]\subseteq\mathcal{M}_{\mathscr{B}}.
\]
Consequently,
\[
\big|[1,H_i]\cap\mathcal{M}_{\mathscr{B}}\big|
\geq
\big|[1,H_i]\cap\mathcal{M}_{\mathscr{E}}\big|-G_i.
\]
Using~\eqref{eq:recovery-points} and letting $i\to\infty$ gives
\[
\overline{d}(\mathcal{M}_{\mathscr{B}})\geq\alpha.
\]
Together with~\eqref{eq:MB-subset-ME}, this yields
\begin{equation}\label{eq:upper-density-match}
\overline{d}(\mathcal{M}_{\mathscr{B}})=\alpha=\overline{d}(\mathcal{M}_{\mathscr{E}}).
\end{equation}

It remains to identify the tautification. Let $C_\mathscr{B}$ be the set from~\eqref{skaleC}. By construction, $e\mathscr{A}_e\subseteq\mathscr{B}$ with $\mathscr{A}_e$ Behrend for every $e\in\mathscr{E}$, and hence $\mathscr{E}\subseteq C_\mathscr{B}$. Moreover, $\mathscr{B}\subseteq\mathcal{M}_{\mathscr{E}}$ and $\mathscr{E}$ is taut, hence primitive, so
\[
(\mathscr{B}\cup\mathscr{E})^{prim}=\mathscr{E}.
\]
Proposition~\ref{pr:9} therefore gives
\begin{equation}\label{eq:tautification-recovery}
\mathscr{B}'=\mathscr{E}.
\end{equation}
By~\eqref{density_difference} and~\eqref{eq:tautification-recovery},
\begin{equation}\label{eq:lower-density-match}
\underline{d}(\mathcal{M}_{\mathscr{B}})
=
\underline{d}(\mathcal{M}_{\mathscr{B}'})
=
\underline{d}(\mathcal{M}_{\mathscr{E}})
=
\delta.
\end{equation}
\begin{sloppypar}
Equations~\eqref{eq:upper-density-match} and~\eqref{eq:lower-density-match} give~\eqref{eq:matching-densities}, and since $\delta<\alpha$, they also show that $\mathscr{B}$ is non-Besicovitch. Finally,~\eqref{eq:tautification-recovery} and~\eqref{eq:positive-gap-nonbes-choice} complete the proof.
\end{sloppypar}
\end{proof}

\begin{remark}\label{ThmA_implies_PropC}
Corollary~\ref{prop2} can be deduced from Theorem~\ref{taut_bes} in the following way. Let $\mathscr{D}$ be any non-Besicovitch taut set.
By Theorem~\ref{taut_bes}, there exists a Besicovitch set
$\mathscr{B}$ with $\mathscr{B}'=\mathscr{D}$.
Then
\[
\overline{d}(\mathcal{M}_\mathscr{D}\setminus \mathcal{M}_\mathscr{B})=\overline{d}(\mathcal{M}_\mathscr{D})-d(\mathcal{M}_\mathscr{B})=\overline{d}(\mathcal{M}_\mathscr{D})-\underline{d}(\mathcal{M}_\mathscr{D})>0
\]
(in the last equality we use \eqref{density_difference}). However, the proof of Theorem~\ref{taut_bes} is quite involved and Corollary~\ref{prop2} has another shorter, more direct proof.
\end{remark}

\subsection{Besicovitch tautification}
The proof of Theorem~\ref{prop1} has two steps. Theorem~\ref{prop11} constructs a primitive thin set $\mathscr{E}$ of positive upper Banach density. We then apply Theorem~\ref{jjj} and Corollary~\ref{n1} to obtain a Besicovitch set $\mathscr{B}$ with $\mathscr{B}'=\mathscr{E}$. The key observation is that $\mathscr{E}\subseteq\mathcal{M}_{\mathscr{B}'}\setminus\mathcal{M}_{\mathscr{B}}$, which yields the required positive upper Banach density.
\begin{theorem}\label{prop11}
There exists a thin (hence Besicovitch) primitive set
$\mathscr{E}=\bigcup_{i\geq 1}\mathscr{E}_i$ with $d^*(\mathscr{E})>0$ such that for $K_i$ large enough and $\mathscr{B}=\bigcup_{i\geq 1}\mathscr{E}_i \mathscr{A}_i$ with $\emptyset\neq\mathscr{A}_i\subseteq [K_i,\infty)$ we have $d^*(\mathcal{M}_\mathscr{E} \setminus \mathcal{M}_\mathscr{B})>0$. If we additionally assume that $\mathscr{A}_i$ are Behrend sets then $\mathscr{B}'=\mathscr{E}$.
\end{theorem}

\begin{proof}
Let $(T_i)_{i\ge 1}$ and $(t_i)_{i\ge 1}$ be increasing sequences of natural numbers satisfying

\begin{equation}\label{eq:tbd11}
((t_i+T_i)!)\divides T_{i+1} \text{ for every }i\geq 1
\end{equation}
and
\begin{equation}\label{eq:tbd21}
\sum_{i=1}^{\infty}\frac{T_i}{t_i}<1.
\end{equation}
Let $J$ denote the union of all sets $[t_i,t_i+T_i-1]$ over all $i\in\N$. By \eqref{eq:tbd21}, $T_{i+1}<t_{i+1}$, while \eqref{eq:tbd11} implies $t_i+T_i\leq T_{i+1}$. Hence $t_i+T_i<t_{i+1}$, so the intervals $[t_i,t_i+T_i-1]$ are pairwise disjoint. Therefore,
\begin{equation}\label{eq:tbd22}
\overline{d}(\cM_{J})\le \sum_{j\in J}\frac{1}{j}\le \sum_{i= 1}^{\infty}\frac{T_i}{t_i}
\end{equation}
and it follows by (\ref{eq:tbd21}) that $J$ is thin. Thus,
\begin{equation}\label{cienki}
\text{any primitive subset of $J$ is taut.}
\end{equation}
Let us denote $\beta:=1-\overline{d}(\cM_{J})$.
By induction on $i$, we construct sets $\mathscr{E}_i\subseteq [t_i,t_i+T_i-1]$ satisfying the following conditions:
\begin{enumerate}[(a)]
\item $\mathscr{E}_{i+1}\subseteq \cF_{\mathscr{E}_1\cup\ldots \cup \mathscr{E}_{i}}$,\label{cond:a}
\item $|\mathscr{E}_i|\ge \beta T_i$\label{cond:b}
\item $\mathscr{E}_i$ is primitive\label{cond:c}
\end{enumerate}
for every $i\in\N$. We set $\mathscr{E}_1=[t_1,t_1+T_1-1]$. Suppose that $\mathscr{E}_1,\ldots,\mathscr{E}_i$ have been constructed. Since the length of the block $[t_{i+1},t_{i+1}+T_{i+1}-1]$ is divisible by the least common multiple of $\mathscr{E}_1\cup\ldots\cup\mathscr{E}_i$ (by (\ref{eq:tbd11})), we have
\begin{equation}
\begin{aligned}
&|\cM_{\mathscr{E}_1\cup\ldots \cup\mathscr{E}_{i}}\cap [t_{i+1},t_{i+1}+T_{i+1}-1]|\\
&\qquad=T_{i+1}\cdot d(\cM_{\mathscr{E}_1\cup\ldots \cup\mathscr{E}_{i}})\\
&\qquad\le T_{i+1}\cdot\overline{d}(\cM_{J})=T_{i+1}(1-\beta).
\end{aligned}
\end{equation}
We set $\mathscr{E}_{i+1}:=[t_{i+1},t_{i+1}+T_{i+1}-1]\setminus \cM_{\mathscr{E}_1\cup\ldots \cup\mathscr{E}_{i}}$. Since $T_i<t_i$, the interval $[t_i,t_i+T_i-1]$ is contained in $[t_i,2t_i)$; hence no two distinct elements of this interval divide one another. Thus $\mathscr{E}_i$ is primitive, verifying condition~\eqref{cond:c}.

Let $\mathscr{E}:=\bigcup_{i=1}^{\infty}\mathscr{E}_i$. By the primitivity of $\mathscr{E}_i$ and by condition~\eqref{cond:a}, $\mathscr{E}$ is primitive. Since $\mathscr{E}\subseteq J$ and $J$ is thin, the set $\mathscr{E}$ is thin; in particular, by~\eqref{cienki}, it is taut. Moreover, by condition~\eqref{cond:b}:
\begin{equation}\label{eq:tbd4}
d^*(\mathscr{E})\ge \beta.
\end{equation}
Let $K_i > \max (\mathscr{E}_1\cup\dots\cup\mathscr{E}_i)$, take any $\emptyset \neq\mathscr{A}_i\subseteq [K_i,\infty)$ and let $\mathscr{B}:=\bigcup_{i\geq 1}\mathscr{E}_i\mathscr{A}_i$. Then
\begin{equation*}
\mathscr{E}_1\cup\dots\cup\mathscr{E}_i \subseteq \mathcal{F}_{\mathscr{A}_i}
\end{equation*}
We claim that $\mathscr{E}\cap\mathcal{M}_{\mathscr{B}}=\emptyset$. Indeed, if $e\in\mathscr{E}_j$ and $e_i a\mid e$ for some $e_i\in\mathscr{E}_i$ and $a\in\mathscr{A}_i$, then for $i\geq j$ we have $a\geq K_i>e$, a contradiction, while for $i<j$ we have $e_i\mid e$, contradicting condition~\eqref{cond:a}. Hence $\mathscr{E}\subseteq \mathcal{M}_\mathscr{E}\setminus \mathcal{M}_\mathscr{B}$, and by~\eqref{eq:tbd4}, $d^*(\mathcal{M}_\mathscr{E}\setminus \mathcal{M}_\mathscr{B})\geq \beta>0$. If we assume additionally that $\mathscr{A}_i$ are Behrend sets then it follows by Theorem~\ref{pr:9_cor} that $\mathscr{B}'=\mathscr{E}$.
\end{proof}

\begin{proof}[Proof of Theorem~\ref{prop1}]
Let $\mathscr{E}$ be the set constructed in the proof of Theorem~\ref{prop11}. By construction, $\mathscr{E}$ is thin and primitive, and by~\eqref{eq:tbd4} we have
\[
d^*(\mathscr{E})\geq\beta>0.
\]
Enumerate $\mathscr{E}=\{e_i:i\geq1\}$. Since $\mathscr{E}$ is primitive and thin, it is taut. Apply Theorem~\ref{jjj} and Corollary~\ref{n1}, choosing for each $i$ a Behrend set
\[
\emptyset\neq\mathscr{A}_i\subseteq\mathcal{P}_i\cap[K_i,\infty)
\]
(for instance, the full tail $\mathcal{P}_i\cap[K_i,\infty)$). Then
\[
\mathscr{B}:=\bigcup_{i\geq1}e_i\mathscr{A}_i
\]
is Besicovitch and $\mathscr{B}'=\mathscr{E}$. In particular, $\mathscr{B}'$ is Besicovitch because $\mathscr{E}$ is thin.

It remains to verify the positive upper Banach density of the difference. We have
\[
\mathscr{E}\cap\mathcal{M}_{\mathscr{B}}=\emptyset.
\]
Indeed, if $e\in\mathscr{E}$ and $e_i a\mid e$ for some $a\in\mathscr{A}_i$, then $e_i\mid e$. Since $\mathscr{E}$ is primitive, this forces $e_i=e$, and then $a\mid1$, a contradiction. Therefore
\[
\mathscr{E}\subseteq\mathcal{M}_{\mathscr{E}}\setminus\mathcal{M}_{\mathscr{B}}
=\mathcal{M}_{\mathscr{B}'}\setminus\mathcal{M}_{\mathscr{B}},
\]
and hence, by~\eqref{eq:tbd4},
\[
d^*(\mathcal{M}_{\mathscr{B}'}\setminus\mathcal{M}_{\mathscr{B}})
\geq d^*(\mathscr{E})\geq\beta>0.
\]
This proves Theorem~\ref{prop1}.
\end{proof}

\begin{remark}\label{invariant_measures_example}
For the Besicovitch set $\mathscr{B}$ used in the proof of Theorem~\ref{prop1} above, the systems $X_\eta$ and $X_{\eta'}$ have exactly the same invariant measures. Let $\mathscr{E}$ be the set constructed in the proof of Theorem~\ref{prop11}; the proof of Theorem~\ref{prop1} chooses $\mathscr{B}$ so that $\mathscr{B}'=\mathscr{E}$. We claim that $\mathscr{E}$ contains an infinite pairwise coprime subset. Suppose that $\mathscr{E}_1,\ldots,\mathscr{E}_i$ have already been constructed and let $Q_i=\lcm(\mathscr{E}_1\cup\cdots\cup\mathscr{E}_i)$. Since $\mathscr{E}_1\cup\cdots\cup\mathscr{E}_i\subset [1,t_i+T_i)$, we have $Q_i\mid (t_i+T_i)!\mid T_{i+1}$ by~\eqref{eq:tbd11}. Hence the interval $[t_{i+1},t_{i+1}+T_{i+1}-1]$ contains an integer $a_{i+1}\equiv 1\pmod{Q_i}$. Such an $a_{i+1}$ is coprime to every element of $\mathscr{E}_1\cup\cdots\cup\mathscr{E}_i$, and therefore, by the definition of $\mathscr{E}_{i+1}$, we have $a_{i+1}\in\mathscr{E}_{i+1}$. Choosing one such element at each stage yields an infinite pairwise coprime subset of $\mathscr{E}$.

Consequently, $1\in D_\mathscr{E}$ (see~\eqref{skaleD} with $\mathscr{E}$ in place of $\mathscr{B}$), and hence $\mathscr{E}^*=\{1\}$. Since $\mathscr{B}'=\mathscr{E}$ and $(\mathscr{B}')^*=\mathscr{B}^*$, it follows that $\mathscr{B}^*=\{1\}$ and thus $\eta^*=\mathbf{0}$, which is a regular Toeplitz sequence. Therefore, by~\cite[Theorem~C]{DKPS},
\[
\mathscr{M}(X_\eta)=\mathscr{M}(X_{\eta'}).
\]
In particular, the two systems have the same ergodic invariant measures. Moreover, since $\mathscr{B}'$ is taut, the Mirsky measure $\nu_{\eta'}$ has full support on $X_{\eta'}$~\cite{KellerTautness}. Hence the common measure center of $X_\eta$ and $X_{\eta'}$ (the closure of the union of the supports of their invariant Borel probability measures) is $X_{\eta'}$.
Thus, in the example from Theorem~\ref{prop1}, positive upper Banach density of $\mathcal{M}_{\mathscr{B}'}\setminus\mathcal{M}_{\mathscr{B}}$ does not produce any additional invariant or ergodic invariant measures.
\end{remark}
\section{A correction concerning \cite{MR3989121}}\label{misquotation}

In [2, Theorem 2.25], it was incorrectly asserted
that the tautification of a Besicovitch set is
again Besicovitch. This is the converse of (1)
and is disproved by Theorem A.
This incorrect assertion was used in the proof
of [2, Theorem 2.18], which states that for every
Besicovitch set
$\mathscr B\subseteq\mathbb N\setminus\{1\}$
there exists a set $D\subseteq\mathcal F_{\mathscr B}$
with
\[
d(\mathcal F_{\mathscr B}\setminus D)=0
\]
such that, for every $r\in\mathbb N$, the
following conditions are equivalent:
$r\in D$;
$\mathcal F_{\mathscr B}-r$ is divisible;
and $\mathcal F_{\mathscr B}-r$ is an averaging
set of polynomial multiple recurrence
(in the sense of [2, Definition 1.2]).
Here, as in [2, Definition 1.3], a set $R$ is
called divisible if, for every $u\in\mathbb N$,
the natural density $d(R\cap u\mathbb N)$
exists and is positive.
For sets of positive natural density, divisibility
is a necessary condition for averaging polynomial
multiple recurrence, as follows by considering
finite cyclic rotations (see [2, Section 1]).
The conclusion that
$d(\mathcal F_{\mathscr B}\setminus D)=0$
is false in general.

Indeed, let $\mathscr C$ be a non-Besicovitch taut set.
By Theorem A, there exists a Besicovitch set $\mathscr B$
such that $\mathscr B'=\mathscr C$. Put
\[
E:=\mathcal F_{\mathscr B}
   \setminus\mathcal F_{\mathscr C}.
\]
As explained in Remark 6.3,
\[
\overline d(E)
=\overline d(\mathcal M_{\mathscr C})
 -\underline d(\mathcal M_{\mathscr C})>0,
\]
while $\delta(E)=0$.
Since $\mathscr C$ is taut and hence not Behrend,
we also have $d(\mathcal F_{\mathscr B})>0$.

For every $r\in E$, there exists $c\in\mathscr C$
dividing $r$. Since
\[
\mathcal F_{\mathscr B}\cap c\mathbb N\subseteq E,
\]
and $\mathcal F_{\mathscr B}\cap c\mathbb N$ has
natural density by the Besicovitch property of
$\mathscr B$, this density must be zero.
As $c\mid r$, it follows that
\[
d\bigl((\mathcal F_{\mathscr B}-r)
       \cap c\mathbb N\bigr)=0.
\]
Thus $\mathcal F_{\mathscr B}-r$ is not divisible
and hence is not an averaging set of polynomial
multiple recurrence. Consequently, every $r\in E$
is an exceptional shift, and the exceptional set
has positive upper natural density, contrary to
[2, Theorem 2.18].

A correction of the affected statements in [2]
will be addressed separately. None of the
proofs in the present paper depends on them.
\section*{Acknowledgements}
We would like to thank both referees for their careful reading of the manuscript and for their valuable and constructive comments. Their suggestions helped us improve the organization, precision, and exposition of the paper.
\section*{Declaration of AI use}

During the revision of this manuscript in response to the referees' comments, the authors used ChatGPT (OpenAI) to assist with improving the exposition and checking the consistency of mathematical arguments. All AI-assisted suggestions were critically reviewed and verified by the authors, who take full responsibility for the final manuscript.

\end{document}